\documentclass{amsart}
\usepackage[all]{xy}
\usepackage{amssymb,latexsym, amscd, amsthm, amsmath, upgreek, bm, dsfont, mathrsfs}
\newtheorem{theorem}{Theorem}[section]
\newtheorem{lemma}[theorem]{Lemma}
\newtheorem{corollary}[theorem]{Corollary}
\newtheorem{proposition}[theorem]{Proposition}
\newtheorem{conjecture}[theorem]{Conjecture}

\theoremstyle{definition}
\newtheorem{definition}[theorem]{Definition}

\theoremstyle{remark}
\newtheorem{remark}[theorem]{\bf Remark}

\newcommand{\zl}{\mathbb Z_l}
\newcommand{\ql}{\mathbb Q_l}
\newcommand{\C}{\mathbb C}

\newcommand{\Q}{\mathbb Q}
\newcommand{\G}{\mathbb G}
\newcommand{\Z}{\mathbb Z}
\newcommand{\R}{\mathbb R}

\newcommand{\N}{\mathbb N}
\newcommand{\F}{\mathbb F}

\newcommand{\CO}{\mathcal O}
\newcommand{\blambda}{{\boldsymbol{\lambda}}}
\newcommand{\wlambda}{\widehat{\lambda}}

\begin{document}

\title[$K$--theory of totally real and CM number fields]{Hecke characters and
the $K$--theory of totally real and CM number fields}
%\pagestyle{headings}

%    Information for first author
\author[G. Banaszak]{Grzegorz Banaszak*}
\address{Department of Mathematics and Computer Science, Adam Mickiewicz University,
Pozna\'{n} 61614, Poland}
\email{banaszak@amu.edu.pl}

%    Information for second author
\author[C. Popescu]{Cristian Popescu**}
\address{Department of Mathematics, University of California, San Diego, La Jolla, CA 92093, USA}
\email{cpopescu@math.ucsd.edu}

%    General info
\date{}
\thanks{*Partially supported by research grant NCN  2013/09/B/ST1/04416 (National Center for Science of Poland.)\newline\indent
**Partially supported by NSF grants
DMS-901447 and DMS-0600905}

\begin{abstract} Let $F/K$ be an abelian extension of number fields with $F$ either CM or totally real and $K$
totally real. If $F$ is CM and the
Brumer-Stark conjecture holds for $F/K$, we construct a family of $G(F/K)$--equivariant
Hecke characters for $F$ with infinite type equal to a special value of
certain $G(F/K)$--equivariant $L$--functions.
Using results of Greither--Popescu \cite{GP} on the Brumer--Stark conjecture we
construct $l$--adic imprimitive versions of these characters, for primes $l > 2$.
Further, the special values of these $l$--adic Hecke characters are used to construct
$G(F/K)$--equivariant Stickelberger--splitting maps in the Quillen localization sequence for $F$, extending the results obtained in \cite{Ba1} for $K=\Q$. We also apply the Stickelberger--splitting maps to construct special
elements in $K_{2n}(F)_l$ and analyze the Galois module structure of the group $D(n)_l$ of divisible elements in $K_{2n}(F)_l$. If $n$ is odd, $l\nmid n$,  and $F=K$ is a fairly general totally real number field, we study the cyclicity of $D(n)_l$ in relation to the classical conjecture of Iwasawa on class groups of cyclotomic fields and its potential generalization to a wider class of number fields. Finally, if $F$ is CM, special values of our $l$--adic Hecke characters are used to construct Euler systems in odd $K$-groups $K_{2n+1}(F, \Z/l^k)$.
These are vast generalizations of Kolyvagin's Euler system of Gauss sums \cite{Rubin}
and of the $K$--theoretic Euler systems constructed in \cite{BG1} when $K=\Q$.
\end{abstract}

\maketitle

%-----------Introduction------------------

\section{Introduction}

\noindent {\bf Notation.}

Let $L$ be a number field. For a nontrivial $\CO_L$--ideal ${\bf a},$
we let, as usual, $N{\bf a}:=|\CO_L/{\bf a}|$ denote the norm of ${\bf a}$ and $\text{Supp}(\bf a)$ denote
the set of distinct prime $\CO_L$--ideals which divide ${\bf a}$. If $M/L$ is a finite abelian extension and the $\CO_L$--ideal ${\bf a}$ is prime then $G_{\bf a}$ denotes
the decomposition group associated to ${\bf a}$, viewed as a subgroup of $G(M/L)$. Further, if  ${\bf a}$ is (not necessarily prime) but coprime to the conductor of $M/L$,
then $\sigma_{\bf a}$ denotes the Frobenius element associated to ${\bf a}$ in $G(M/L)$. We will let $S_{\infty}(L)$ denote the set of archimedean primes of $L$, $I_L$ the group of
fractional $\CO_L$--ideals and $I_L({\bf a})$ the group of fractional $\CO_L$ ideals which are coprime to ${\bf a}$.
\medskip

For a prime number $l$ the symbol $\omega_{L,l}$ denotes the $l$--adic cyclotomic character:
$$\omega_{L,l}: G_L\to\zl^\times.$$
Recall that $G_L :=G(\overline L/L)$ acts on the $\Z_l$-modules $\Z_l (n),$ $\Q_l (n)$ and
$\Q_l/\Z_l(n) := \Z_l (n) \otimes_{\Z_l} \Q_l/\Z_l = \Q_l (n)/ \Z_l (n)$  via the $n$--th power
of $\omega_{L,l}.$
\medskip

For $n\geq 1$, let $w_n(L)_l := |(\Q_l/\Z_l)(n)^{G_L}|$ and
$$w_n(L):=\prod_{l \geq 2} \,  w_n(L)_l.$$
Note that
$w_1(L)_l=|\mu_L\otimes_{\bf Z}\Z_l|$, where $\mu_L$ is the group of roots of unity in $L$.
For simplicity, we let $w_L:=w_1(L)=|\mu_L|$.
\medskip

If $A$ is an abelian group and $k\in\Z_{\geq 0}$ we let $A[l^k]$ denote the $l^k$--torsion subgroup of $A$ and
$A_l:= A \otimes_{\Z}\Z_l.$
\medskip

For a unital ring $R$ and an integer $m\geq 0$, $K_m(R)$ and $K_m(R, \Z/l^k)$ denote the corresponding
Quillen $K$--group and $K$--group with coefficients in $\Z/l^k$, respectively.

\bigskip

\subsection{Hecke characters} In this paper, we consider abelian extensions $F/K$ of number fields, where $F$ is either CM or totally real and $K$ is totally real. We consider two $\CO_K$--ideals ${\bf f}$ and
${\bf b}$, such that ${\bf f}$ is divisible by the (finite) conductor of $F/K$ and ${\bf b}$ is coprime to ${\bf f}$. To the data $(F/K, {\bf f}, {\bf b})$ one associates the Galois equivariant holomorphic $L$--function:
$$\Theta_{\bf f, b}: \C\to \C[G],$$
$$\Theta_{\bf f, b}(s):=(1-N{\bf
b}^{1+s} \cdot\sigma_{\bf b}^{-1}) \cdot \sum_{\sigma\in G(F / K)}\zeta_{\bf f}(\sigma,
s)\cdot\sigma^{-1}.$$
  The special values $\Theta_n({\bf b, f}):= \Theta_{\bf f, b}(-n)$,
 for all $n\in\Bbb Z_{\geq 0}$, are what Coates \cite{C} calls higher Stickelberger elements. According to a
 deep theorem of Deligne--Ribet \cite{DR},
 $$\Theta_n({\bf b, f})\in \Z[G(F/K)], \qquad\text{ for all } n\in\Z_{\geq 0},$$
 as long as ${\bf b}$ is coprime to $w_{n+1}(F)$. In particular, if  ${\bf b}$ is coprime to $w_F$, then $$\Theta_0({\bf b, f})\in\Z[G(F/K)].$$
Now, let us assume that $F$ is CM. An equivalent formulation of
the Brumer-Stark conjecture (see \cite{Popescu-PCMI}) is the following.

\begin{conjecture}[BrSt$(F/K, {\bf f})$, Brumer-Stark] Let $F/K$ and ${\bf f}$ be as above. Then, for any prime $\CO_K$--ideal ${\bf b}$ which is coprime to $w_F\cdot{\bf f}$, we have
$$\Theta_0({\bf b, f})\in{\rm Ann}_{\Z[G(F/K)]}{\rm CH}^1(F)_{T_{\bf b}}^0, $$
where ${\rm CH}^1(F)_{T_{\bf b}}^0$ is the Arakelov class--group associated to $(F, {\bf b})$ defined in \S3.1.
\end{conjecture}
\noindent Under the assumption that the Brumer--Stark conjecture holds for $(F/K, {\bf f})$, in \S3.1 (see Proposition \ref{Hecke-characters}) we construct $G(F/K)$--equivariant Hecke characters
$$\lambda_{\bf b, f}: I_F({\bf b})\to F^\times, $$
of conductor ${\bf b}$ and infinite type $\Theta_0({\bf b, f})$, for all $\CO_K$--ideals ${\bf b}$ coprime
to $w_F\cdot{\bf f}$.  For ${\bf a}\in I_F({\bf b})$,
the value $\lambda_{\bf b, f}({\bf a})$ is the unique element in $F^\times$ with Arakelov divisor
$${\rm div}_F(\lambda_{\bf b, f}({\bf a}))=\Theta_0({\bf b, f})\cdot{\bf a}$$
and with some additional arithmetic properties (see \S3.1 for details).

\medskip
Weil's Jacobi sum Hecke characters \cite{Weil} can be recovered from our construction for $K=\Bbb Q$.
The values of our characters $\lambda_{\bf b, f}$ are generalizations of the classical Gauss sums, which arise in Weil's construction.
If viewed in towers of abelian CM extensions of a fixed $F/K$, these values satisfy norm compatibility relations (see Lemma \ref{norm-lemma})
which lead to Euler systems for the algebraic group $\Bbb G_m$, generalizing the Euler system of Gauss sums of Kolyvagin
and Rubin \cite{Rubin} (see Lemma \ref{norm-lemma} and Remark \ref{Euler-system-remark}). It is also worth remarking that our construction of Hecke characters is somewhat more refined (in that it
keeps track of conductors and Galois theoretic and arithmetic properties of special values) than that carried out by Yang in \cite{Yang}, following ideas of Hayes \cite{Hayes}.
\medskip

The Brumer--Stark conjecture is not fully proven yet. However, the results
of Greither--Popescu \cite{GP} show that if the classical Iwasawa invariant $\mu_{F,l}$
vanishes (as conjectured by Iwasawa), then an $l$--adic, imprimitive version of BrSt$(F/K, {\bf f})$ holds, for all
primes $l>2$ (see Theorem \ref{GP} below for the precise result). As a consequence, under the vanishing hypothesis above, in \S3.2 (see Lemma \ref{GP-Hecke-characters}) we construct, for all primes $l>2$, the
$l$--adic $\Z_l[G(F/K)]$--equivariant versions
$$\lambda_{{\bf b}, {\bf f}}: I_F({\bf b})\otimes\Z_l\to F^\times\otimes\Z_l$$
of the Hecke characters above, provided that ${\bf f}$ is divisible by all $l$--adic primes in $K$ (an imprimitivity condition.)
These $l$--adic, imprimitive Hecke characters are sufficient for our applications to $K$--theory in this paper.

\subsection{Euler Systems in odd $K$--theory with coefficients} In the case where $F$ is CM, in \S7 we push our generalization of Gauss sums farther and use the values of our ($l$--adic, imprimitive) Hecke characters
along with Bott elements to construct Euler systems for the odd $K$--theory $K_{2n+1}(F, \Z/l^k)$ with coefficients in  $\Z/l^k$,
for all $n\geq 0$ and all primes $l>2$. For $n=0$ and $K=\Q$ one recovers the Euler System of Gauss sums (modulo $l^k$) of Kolyvagin and
Rubin \cite{Rubin}.
For $n\geq 1$ and $K=\Q$ one recovers the $K$--theory Euler systems constructed in \cite{BG1}.
(See Theorem \ref{Theorem ES2} and Remark \ref{Euler-system-remark} for details.)

\subsection{Stickelberger-splitting maps in $K$--theory} Assume that $F$ is either CM or totally real, fix a prime $l>2$, an integer $n\geq 1$ and $\CO_K$--ideals ${\bf b}$ and ${\bf f}$ as above. We consider the
$l$--torsion part of the Quillen localization sequence (\cite{Q1} and \cite{So1})
\begin{equation}\label{Quillen localization sequence}\xymatrix{0\ar[r] &K_{2n}(\CO_F)_l\ar[r] &K_{2n}(F)_l\ar[r]^{\partial_F\quad} &\bigoplus_{v}K_{2n-1}(k_v)_l\ar[r] &0. }\end{equation}
Above, $v$ runs over all the maximal ideals of $\CO_F$ and $k_v$ is the residue field of $v$. In \S4 (see Theorem \ref{lambda-properties}), we use special values of the
$l$--adic imprimitive Hecke characters for extensions $F(\mu_{l^k})/K$, with $k\geq 1$, to construct a morphism
$$\Lambda \,\, : \,\, \bigoplus_{v}K_{2n-1}(k_v)_l\longrightarrow  K_{2n}(F)_l$$
of $\Z_l[G(F/K)]$--modules, such that for all $x \in \bigoplus_{v}K_{2n-1}(k_v)_l$
$$(\partial_F\circ\Lambda)(x)= x^{\l^{v_l(n)}\cdot\Theta_0({\bf b}, {\bf f})}.$$
Following \cite{Ba1}, we call $\Lambda$
a Stickelberger--splitting map for exact sequence \eqref{Quillen localization sequence}. As shown in \cite{Ba1} and \cite{BP} the construction of such maps has far reaching arithmetic applications.
 The main idea behind constructing $\Lambda$ is as follows: For each maximal $\CO_K$--ideal $v_0$ we pick an $\CO_F$--prime $v$ dividing $v_0$ and let $l^k:=|K_{2n-1}(k_v)_l|$.
 Then we pick  a prime $w$ in $E:=F(\mu_{l^k})$ which divides $v$ and consider the special element $\lambda^\ast_{\bf b, f}(w)\in E^\times\otimes\Z_l$, where
 $$\lambda^\ast_{\bf b, f}: I_E\otimes\Z_l\to E^\times\otimes\Z_l$$ is a carefully chosen $\Z_l[G(E/K)]$--linear extension to the group $I_E$ of all fractional ideals in $E$ of the $l$--adic Hecke character $\lambda_{\bf b, f}$ associated to the data $(E/K, {\bf b}, {\bf f}).$  Next, we map
 $\lambda^\ast_{\bf b, f}(w)$ into $E^\times/E^{\times l^k}\simeq K_1(E, \Z/l^k)$. Then we construct the special element:
 $$Tr_{E/F}(\lambda^\ast_{\bf b, f}(w)\ast b(\beta(\xi_{l^k})^{\ast n}))^{\gamma_l} \in K_{2n}(F)[l^k],$$
where $\xi_{l^k}$ is a generator of $\mu_{l^k}$ in $E,$ $\beta(\xi_{l^k}) \in K_2(E, \Z/l^k)$ is the
corresponding Bott element, $\gamma_l\in\Z_l[G(F/K)]$ is an exponent defined in \eqref{gamma_l} and \newline
$${b: K_{2n+1}(E, \Z/l^k)\to K_{2n}(E)[l^k]}$$ is the Bockstein homomorphism. Consequently, for any generator
$\xi_v$ of $K_{2n-1}(k_v)_l$ there is a unique group morphism
$$\Lambda_v: K_{2n-1}(k_v)_l\to K_{2n}(F)_l,$$
$$\Lambda_v(\xi_v)=Tr_{E/F}(\lambda^\ast_{\bf b, f}(w)\ast b(\beta(\xi_{l^k})^{\ast n}))^{\gamma_l}.$$
For a carefully chosen generator $\xi_v$ (see Definition \ref{picking xi_v}) the map $\Lambda_v$ satisfies
$$\partial_F\circ \Lambda_v(\xi_v)=\xi_v^{\l^{v_l(n)}\cdot\Theta_0({\bf b}, {\bf f})}$$
and is $\Z_l[G_v]$--equivariant (see Theorem \ref{local-lambda-theorem}).
The map $\Lambda$ is the unique $\Z_l[G(F/K)]$--linear morphism which equals
$\Lambda_v$ when restricted to  $K_{2n-1}(k_v)_l$, for all the chosen $\CO_F$--primes $v$ (see Theorem \ref{lambda-properties}).

The above construction generalizes to arbitrary totally real fields $K$ the construction of \cite{Ba1}
done in the case $K=\Q$. The above construction is very different from that in \cite{BP} and it has the advantage of being $\Z_l[G(F/K)]$--linear unlike the one in loc.cit.,
a property which leads to new arithmetic applications, as shown below.

\subsection{Divisible elements in $K$--theory and Iwasawa's conjecture}

For a number field $L$ and an $n > 0$ the group of divisible elements in $K_{2n}(L)_l$ is given by
$$div(K_{2n}(L)_l):=\bigcap_{k>0}K_{2n}(L)_l^{\,l^k}.$$
It is well known that the groups $div(K_{2n}(L)_l)$ are contained in $K_{2n}(\CO_L)_l$ and they are
the correct higher $K$--theoretic analogues of the ideal--class group $Cl(\CO_L)_l = (K_0(O_L)_{tor})_l$ (see \cite{Ba2} and \cite{GP}, for example). The group
$div(K_{2n}(L)_l)$ is also one of the main obstructions (see \cite[Section 4 and Theorem 6.4]{Ba3}) to the splitting of exact sequence \eqref{Quillen localization sequence} (in the category of $\Z_l$--modules).
In particular $div(K_{2n}(\Q)_l)$ is the only obstruction
\cite[Corollary 6.6]{Ba3} to the splitting of \eqref{Quillen localization sequence}
for $L = \Q.$ Combined with the newly proved Quillen--Lichtenbaum
Conjecture \cite{Voevodsky} and with \cite[Theorem 5.10]{Ba3}, Theorem 4, p. 299 in \cite{Ba2} can be restated as
$$[K_{2n}(\CO_L)_l: div(K_{2n}(L)_l)]= \frac{\prod_{v|l}w_{n}(L_v)_l}{w_{n}(L)_l},$$
for all number fields $L$ and all $n\geq 1$, where $v$ runs over all the $l$--adic primes of $L$ and $L_v$ is the $v$--adic completion of $L$. Thus, for $L=\Bbb Q$ and all $n\geq 1$ we have
$$K_{2n}(\Z)_l=div(K_{2n}(\Q)_l).$$

Let $\omega: G(\Q(\mu_l)/\Q)\to \Z_l^\times$ denote the  Teichmuller character.
Using divisible elements (see \cite{BG1} and \cite{BG2}), one of Kurihara's results in \cite{Ku} can be restated as
follows:
\begin{equation}\label{Kurihara} div(K_{2n}(\Bbb Q)_l)\text { is cyclic} \,\,\,
\Longleftrightarrow \,\,\, Cl(\Bbb Q(\mu_l))_l^{\, \omega^{-n}} \text{ is cyclic},\end{equation}
for all odd $n\geq 1$ and
\begin{equation}\label{Kurihara2} div(K_{2n}(\Bbb Q)_l) \, = \, 0 \,\,\, \Longleftrightarrow \,\,\,
Cl(\Bbb Q(\mu_l))_l^{\, \omega^{-n}} \, = \, 0,\end{equation}
for all even $n\geq 1.$
The right--hand side of \eqref{Kurihara} is a deep
conjecture of Iwasawa and right--hand side of \eqref{Kurihara2} is
an equaly deep conjecture of Kummer-Vandiver.

From the above remarks it is clear that the study of the groups $div(K_{2n}(L)_l)$ is of central importance for understanding the arithmetic of $L$.
\medskip

In \S5, we use our Stickelberger splitting map to study the $\Z_l[G(F/K)]$--module structure and group structure of the abelian group $div(K_{2n}(F)_l).$ We work in this section under the simplifying hypotheses that  $F(\mu_l)/K$ is ramified at all the $l$--adic primes, $F(\mu_{l^\infty})/F(\mu_l)$ is totally ramified at all these primes and
$l\nmid n\cdot |G(F/K)|$. In this context, we show (see Theorem \ref{divisible elements via Lambda}) that
$$div(K_{2n}(F)_l)^\chi=K_{2n}(O_F)_l^\chi\cap{\rm Im}(\Lambda), $$
for all irreducible $\C_l$--valued characters $\chi$ of $G$, such that $\chi(\Theta_n({\bf b, f}))\ne 0.$ In the particular case $F=K$ and $n\geq 1$ odd, this implies
that
\begin{equation}\label{div in terms of Lambda} div(K_{2n}(K)_l)=K_{2n}(O_K)_l\cap{\rm Im}(\Lambda).\end{equation}
These results show that the divisible elements are in fact special values of our maps $\Lambda$ and can be explicitely constructed,
as explained above, out of special values of our $l$--adic Hecke characters and Bott elements. Further considerations based on \eqref{div in terms of Lambda}
lead us to the proof of the following equivalence (see Theorem \ref{generalize Iwasawa}):
\begin{equation}\label{injectivity} div(K_{2n}(F)_l) \text{ is cyclic } \Longleftrightarrow \Lambda_{v_0} \text{ is injective, }\end{equation}
for a well chosen ideal ${\bf b}$ and $\CO_F$--prime $v_0$, assuming that
$\prod_{v|l}w_{n}(F)_l=1$.

In particular, in \S6 we combine \eqref{injectivity} for $F =\Q$ with our explicit construction of $\Lambda$ to obtain a
new proof of Kurihara's result \eqref{Kurihara} (see Theorem \ref{another look at the Iwasawa conj.}.)  It is hoped that the techniques
developed in \S6 can be extended to other totally real fields and to a generalization of Iwasawa's conjecture in that context.
\bigskip

\noindent
{\bf  Acknowledgments.}
The first author would like to thank the University
of California, San Diego for its hospitality and financial
support during the period December 2010--June 2011. The second author would like
to thank the Banach Center (B{\c e}dlewo, Poland) for hosting him in July 2012.

%-----------Brummer Stark and Stickelberger elements------------------

\noindent
\section{The higher Stickelberger elements}

 Let $F/K$ be a finite, abelian CM or totally real extension of a totally real
number field $K.$ Let ${\bf f}$ be the (finite) conductor of $F / K$ and let ${\bf f'}$
be any nontrivial $\mathcal O_K$--ideal divisible by all the primes dividing ${\bf f}$, i.e. ${\rm Supp}(\bf f)\subseteq {\rm Supp}(\bf f').$

For all ${\bf f'}$ as above and all $\sigma\in G(F/K)$, let $\zeta_{\bf f'}(\sigma, s)$ is the
$\bf f'$--imprimitive partial zeta function associated to
$\sigma\in G(F/K)$ of complex variable $s$. For $\Re(s)>1$, this is defined by the absolutely and compact-uniformly convergent series
$$\zeta_{\bf f'}(\sigma, s);=\sum_{\bf a} N{\bf a}^{-s},$$
where the sum is taken over all the ideals ${\bf a}$ of $\mathcal O_K$ which are coprime to ${\bf f}'$ and
such that $\sigma_{\bf a}=\sigma$. It is well known that $\zeta_{\bf f'}(\sigma, s)$ has a unique meromorphic continuation
to the entire complex plane $\Bbb C$ and that it is holomorphic away from $s=1$.

\begin{definition}[Coates, \cite{C}]
For all $n\in\Bbb Z_{\geq 0}$, all ${\bf f'}$ as above, and all $\mathcal O_K$--ideals $\bf b$ coprime to $\bf f'$,
the Stickelberger elements
$\Theta_{n} ({\bf b}, {\bf f'}) \in \C [G(F / K)]$ are given by
$$\Theta_{n} ({\bf b}, {\bf f'}):=(1-N{\bf
b}^{1+n}\cdot\sigma_{\bf b}^{-1})\cdot \sum_{\sigma\in G(F / K)}\zeta_{\bf f'}(\sigma,
-n)\cdot\sigma^{-1}.$$
\end{definition}
\begin{remark}\label{support} Note that $\zeta_{\bf f'}(\sigma, s)$ and consequently the elements $\Theta_{n} ({\bf b}, {\bf f'})$ only depend on
${\rm Supp}(\bf f')$ and not on $\bf f'$ per se.
\end{remark}
A deep theorem of Siegel (see \cite{Siegel}) implies that
$\Theta_{n} ({\bf b}, {\bf f'}) \in \Q [G(F / K)]$, for all $\bf f'$, $\bf b$
and $n$ as above. In \cite{DR}, Deligne and Ribet
proved the following refinement of Siegel's theorem.

\begin{theorem}[Deligne-Ribet, \cite{DR}] Let $\bf f'$, $\bf b$ and $n$ be as above and let
$l$ be a prime number. Then, we have
$$\Theta_{n} ({\bf b}, {\bf f'}) \in \Z_l[G(F / K)],$$ as
long as $\bf b$ is coprime to $w_{n+1}(F)_l$.
\end{theorem}

\noindent
Consequently, if ${\bf b}$ and ${\bf f'}$ are as above and ${\bf b}$ is coprime to $w_F$, then we have
$$\Theta_0({\bf b}, {\bf f'})\in\Z[G(F/K)].$$

A fundamental congruence relation between $\Theta_0({\bf b}, {\bf f'})$ and $\Theta_n({\bf b}, {\bf f'})$, for arbitrary $n\geq 1$,
is proved in \cite{DR}. In order to state it, let us first note that for every $n\geq 1$ the character $\omega_{K,l}^n$ modulo $w_n(F)_l$
factors through $G(F/K)$. Consequently, we obtain a group morphism
$$\omega_{K,l}^{(n)} :G(F/K)\to (\zl/w_n(F)_l)^\times.$$
This extends to a unique $\zl/w_n(F)_l$--algebra isomorphism
$$t_{n}:(\zl/w_n(F)_l)[G(F/K)]\simeq(\zl/w_n(F)_l)[G(F/K)],$$ which sends  $\sigma\to\omega_{K,l}^{(n)}(\sigma)^{-1}\cdot\sigma,$
for all $\sigma\in G(F/K)$.
\begin{theorem}[the Deligne--Ribet congruences, \cite{DR}]\label{Deligne-Ribet-congruences} For all $\bf f'$, $\bf b$, $l$ and $n\geq 1$ as above, if
$\bf b$ is coprime to $w_{n+1}(F)_l$, then we have
$$\widehat{\Theta_n({\bf b}, {\bf f'})}=t_n(\widehat{\Theta_0({\bf b}, {\bf f'})}),$$
where $\widehat x$ is the class of $x$ modulo $w_n(F)_l$, for all $x\in\zl[G(F/K)]$.
\end{theorem}
\begin{remark}\label{L-function-theta}
It is easily seen that for all ${\bf b}$ and ${\bf f'}$ as above we have an equality
$$\Theta_{n} ({\bf b}, {\bf f'}):=(1-N{\bf
b}^{1+n}\cdot\sigma_{\bf b}^{-1})\cdot \sum_{\chi\in\widehat{G(F / K)}}L_{\bf f'}(\chi,
-n)\cdot e_{\overline\chi},$$ where $L_{\bf f'}(\chi, s)$ is the
$\bf f'$--imprimitive $L$--function associated to
the complex, irreducible character $\chi$ of $G(F/K)$ and $e_{\chi}$ is the usual idempotent
associated to $\chi$ in the group algebra $\C[G(F/K)]$.

As usual, if $F$ is a CM field we will call a character $\chi$ of $G(F/K)$ even if $\chi(j)=1$
and odd if $\chi(j)=-1$, where $j$ is the unique complex conjugation automorphism of $F$ (contained in $G(F/K)$.)
If $F$ is totally real, then all characters $\chi$ of $G(F/K)$ are called even. A well known consequence of the functional equation
for $L_{\bf f'}(\chi, s)$ is that its order of vanishing at $s=0$ is given by
the following formula
$${\rm ord}_{s=0}L_{\bf f'}(\chi, s)=\left\{
                                       \begin{array}{ll}
                                         {\rm card\,}\{w\in {\rm Supp}({\bf f'})\cup S_\infty(F)\mid \chi(G_w)=1\}, & \hbox{if $\chi\ne\mathbf 1$;} \\
                                         {\rm card\,\,}{\rm Supp}({\bf f'})+{\rm card\,}S_\infty(F)-1, & \hbox{\hbox{if $\chi = \mathbf 1$,}}
                                       \end{array}
                                     \right.$$
where $\mathbf 1$ is the trivial character of $G(F/K)$. The formula above shows that if $\chi$ is an even character, then
$L_{\bf f'}(\chi, 0)=0$. This implies that for all ${\bf b}$ and ${\bf f'}$ as above, if $F$ is a CM field we have
\begin{equation}\label{odd-theta}
\Theta_0({\bf b}, {\bf f'})\in (1-j)\cdot\Z[G(F/K)].
\end{equation}
\end{remark}
\medskip
For a more detailed discussion of the higher Stickelberger elements and
their basic properties, the reader can consult for example \cite{BP}, Section 2.

\section{The Brumer--Stark elements and associated Hecke characters}

In the 1970s, Brumer formulated a conjecture which generalizes to a great extent Stickelberger's
classical theorem. This conjecture was rediscovered in a much more precise form by Stark in \cite{Stark},
as a particular case of what we now call the refined, order of vanishing $1$, abelian Stark conjecture. For a very
lucid presentation of the Brumer-Stark conjecture, the reader is strongly advised to consult
Chpt. IV, \S6 of \cite{Tate}. For a more modern presentation, we refer the reader to \S4.3 of \cite{Popescu-PCMI}, where
the second author reformulates the Brumer-Stark conjecture in terms of the annihilation of certain generalized Arakelov class--groups.

\subsection{The global theory.} In what follows, we remind the reader the formulation of the Brumer--Stark conjecture stated in \cite{Popescu-PCMI} in the (slightly more restrictive) context relevant for our current purposes
and use it to derive some useful consequences on the existence of certain so--called Brumer--Stark elements and Hecke characters.
\medskip

Let $F/K$ and ${\bf f'}$ be as in the previous section. Throughout this section $F$ is assumed to be a CM number field. Let $S_\infty$ the set of infinite (archimedean) primes in $F$.
Let $S_{\bf f'}$ be the union of $S_\infty$ with the set consisting of all the primes in $F$ dividing ${\bf f'}$. We consider
the usual Arakelov divisor group associated to $F$:
$${\rm Div}_{S_\infty}(F):=\left (\bigoplus_{w\not\in S_\infty}\Z\cdot w\right )\bigoplus\left(\bigoplus_{w\in S_{\infty}}\R\cdot w\right ), $$
where the sum is taken over primes $w$ in $F$. Further, we define a degree map
$${\rm deg}_{F}:  {\rm Div}_{S_\infty}(F)\to \R$$
to be the unique map which is $\Z$--linear on the first direct summand above and $\R$--linear on the second
and which also satisfies the equalities
$${\rm deg}_{F}(w)=\left\{
                            \begin{array}{ll}
                              1, & \hbox{if $w\in S_\infty$;} \\
                              \log |{\rm N}w|, & \hbox{if $w\not\in S_\infty$.}
                            \end{array}
                          \right.$$
Above, as usual, we let ${\rm N}w:={\rm card}(\mathcal O_F/w)$. We let ${\rm Div}_{S_\infty}^0(F)$
denote the kernel of the group morphism ${\rm deg}_{F}$. The product formula (for the canonically normalized metrics of $F$) permits us to define a
divisor map (which is a group morphism) by
$${\rm div}_{F}: F^\times \to {\rm Div}_{S_\infty}^0(F),
\quad x\to \sum_{w\not\in S_\infty}{\rm ord}_w(x)\cdot w + \sum_{w\in S_\infty}(-\log|x|_w)\cdot w\,,$$
where ${\rm ord}_w(\cdot)$ and $|\cdot|_w$ denote the canonically
normalized valuation and metric associated to $w$, respectively. The Arakelov class group (first Chow group)
associated to $F$ is defined as follows.
$${\rm CH}^1(F)^0:=\frac{{\rm Div}_{S_\infty}^0(F)}{{\rm div}_F(F^\times)}.$$

Next, following \cite{Popescu-PCMI}, we define a generalized version of the above constructions. For that purpose, let
$\frak b$ be a prime ideal in $K$, coprime to ${\bf f}\cdot w_F$. Let $T_{\frak b}$ be the set
of primes in $F$ which sit above $\frak b$. Further, we let $F_{\frak b}^\times$ denote the subgroup of $F^\times$ consisting of those elements which are congruent to $1$ modulo
$\frak b$. We consider the following subgroup of ${\rm Div}_{S_\infty}(F)$:
$${\rm Div}_{S_\infty, T_{\frak b}}(F):=\left (\bigoplus_{w\not\in S_\infty\cup T_{\frak b}}\Z\cdot w\right )\bigoplus\left(\bigoplus_{w\in S_{\infty}}\R\cdot w\right )$$
and let ${\rm Div}_{S_\infty, T_{\frak b}}^0(F):={\rm Div}_{S_\infty, T_{\frak b}}(F)\cap {\rm Div}^0_{S_\infty}(F).$ Obviously, we have ${\rm div}_F(F_{\frak b}^\times)\subseteq {\rm Div}_{S_\infty, T_{\frak b}}^0(F).$
We define the following generalized Arakelov class group
$${\rm CH}^1(F)_{T_{\frak b}}^0:=\frac{{\rm Div}_{S_\infty, T_{\frak b}}^0(F)}{{\rm div}_F(F_{\frak b}^\times)}.$$
For a detailed description of the structure of these (classical or generalized) Arakelov class groups, their links to ideal
class groups and special values of zeta functions, the reader may consults \cite{Popescu-PCMI}, Section 4.3. However, the reader
should be aware at this point that these Arakelov class groups are by no means finite: if endowed with the obvious
topology, their connected component at the origin is compact of volume equal to a certain non-zero (generalized) Dirichlet regulator
and their group of connected components is isomorphic to a certain (ray) class group.

\begin{remark}\label{injective-div} Obviously, we have $\ker ({\rm div}_F) = \mu_F$, so the map ${\rm div}_F$ factors through $F^\times/\mu_F$. In what follows,
we abuse notation and use ${\rm div}_F$ to denote the factor map as well.
However, for a prime ideal $\frak b$ coprime to $w_F$, we have $F_{\frak b}^\times \cap\mu_F=\{1\}$. This makes
the divisor morphism ${\rm div}_{F}$ {\it injective when restricted to $F_{\frak b}^\times$.}

Finally, let us observe that the groups $F^\times$, $F_{\mathfrak b}^\times$,
${\rm Div}_{S_\infty}^0(F)$ and ${\rm Div}_{S_\infty, T_{\frak b}}^0(F)$  are endowed with natural $\Z[G(F/K)]$--module structures and that the map ${\rm div}_F$ is $G(F/K)$--equivariant.
Consequently, ${\rm CH}^1(F)^0$
and ${\rm CH}^1(F)_{T_{\frak b}}^0$ are endowed with natural $\Z[G(F/K)]$--module structures.
\end{remark}

As we prove in \cite{Popescu-PCMI} (see Proposition 4.3.5(1)), the classical Brumer-Stark conjecture for the set of data
$(F/K, S_{\bf f'})$ is equivalent to the following statement.

\begin{conjecture}[${\rm BrSt}(F/K, S_{\bf f'})$, Brumer-Stark]\label{Brumer-Stark} For all prime ideals ${\frak b}$ in $K$ which are coprime to ${\bf f'}\cdot w_F$, we have
$$\Theta_{0}(\frak b, {\bf f'})\in {\rm Ann}_{\Z[G(F/K)]} {\rm CH}^1(F)_{T_{\frak b}}^0.$$

\end{conjecture}

\begin{remark} In fact, it is sufficient to prove the conjecture above for all but finitely
many prime ideals $\frak b$. Moreover, it is sufficient to prove the statement above for any
(finite) set $\mathcal T$ of prime ideals $\frak b$ which are coprime to ${\bf f'}\cdot w_F$ and
such that the set $\{1-{\rm N}\frak b\cdot\sigma_{\frak b}^{-1}\mid \frak b\in \mathcal T\}$ generates
the $\Z[G(F/K)]$--ideal ${\rm Ann}_ {\Z[G(F/K)]}(\mu_F)$. Also, it is important to note that if one proves
the statement above for a given ${\bf f'}$, then it also follows for any ${\bf f''}$, such that ${\rm Supp}({\bf f'})\subseteq{\rm Supp}({\bf f''})$.
Indeed, this is an immediate consequence of the obvious equality
$$\Theta_{0}(\frak b, {\bf f''})=\prod_{w}(1-\sigma_w^{-1})\,\cdot\, \Theta_{0}(\frak b, {\bf f'}),$$
where the product runs over all the primes $w$ in $K$ dividing ${\bf f''}$ but not dividing ${\bf f'}$. All these
facts are proved in \cite{Popescu-PCMI}, Section 4.3.
\end{remark}

\begin{remark}\label{annihilation-prime} Let us note that if ${\rm BrSt}(F/K, S_{\bf f'})$ holds for a prime $\frak b$, then
$$\Theta_{0}(\frak b, {\bf f'})\in {\rm Ann}_{\Z[G(F/K)]}  Cl(\mathcal O_F)_{\frak b}\subseteq {\rm Ann}_{\Z[G(F/K)]}  Cl(\mathcal O_F),$$
where $Cl(\mathcal O_F)_{\frak b}$ is the ray--class group of conductor ${\frak b}\CO_F$ associated to $F$.
Indeed, this is a direct consequence of the existence of (commuting) natural $\Z[G(F/K)]$--linear surjections
$$\xymatrix{
{\rm CH}^1(F)_{T_{\frak b}}^0 \ar@{>>}[r]\ar@{>>}[d] & Cl(\mathcal O_F)_{\frak b} \ar@{>>}[d] \\
{\rm CH}^1(F)^0\ar@{>>}[r] &Cl(\mathcal O_F),}$$
explicitly constructed in \cite{Popescu-PCMI}, Section 4.3.

However, for our current purposes we need
a more explicit proof and analysis of this consequence of the Brumer--Stark conjecture.
So, let us take an ideal class ${\widehat{\bf c}}$ in $Cl(\mathcal O_F)_{\frak b}$ associated to a fractional $O_F$--ideal
${\bf c}$ which is coprime to ${\mathfrak b}$.  In the obvious
manner, we can associate to ${\bf c}$ a divisor $\widetilde{\bf c}$ in ${\rm Div}_{S_\infty, T_{\frak b}}(F)$.
Now, let us pick an infinite (archimedean) prime $\infty$ of $F$ and note that
$(\widetilde{\bf c}-{\rm deg}_F({\widetilde{\bf c}})\cdot\infty)\in {\rm Div}^0_{S_\infty, T_{\frak b}}(F)$.
Consequently, by our assumption, there is a unique $\lambda_{\frak b, {\bf f}'}(\bf c)\in F_{\frak b}^\times$, such that
$$
\Theta_{0}(\frak b, {\bf f'})\cdot (\widetilde{\bf c}-{\rm deg}_F({\widetilde{\bf c}})\cdot\infty)={\rm div}_F(\lambda_{\frak b, \bf f'}(\bf c)).$$
The uniqueness of $\lambda_{\frak b, \bf f'}(\bf c)$ follows from the injectivity of ${\rm div}_F$ when restricted to $F_{\frak b}^\times$. (See the Remark \ref{injective-div}.)
Note that due to \eqref{odd-theta} and to the obvious fact that $(1-j)\cdot\infty=0$ in ${\rm Div}^0_{S_\infty, T_{\frak b}}(F)$, the last
equality can be rewritten as
\begin{equation}\label{divisors-prime}
\Theta_{0}(\frak b, {\bf f'})\cdot \widetilde{\bf c}={\rm div}_F(\lambda_{\frak b, \bf f'}(\bf c)).
\end{equation}
This shows that the element $\lambda_{\frak b, f'}(\bf c)$ does not depend on the choice of the infinite prime $\infty$.
Moreover, it shows that ${\bf c}^{\Theta_{0}(\frak b, {\bf f'})}$ is a principal $O_F$--ideal generated by
$\lambda_{\frak b, {\bf f'}}(\bf c)$, which proves the annihilation consequence claimed above.
\end{remark}

\begin{remark}\label{annihilation-any} If ${\rm BrSt}(F/K, S_{\bf f'})$ holds, then
$$\Theta_{0}({\bf b}, {\bf f'})\in {\rm Ann}_{\Z[G(F/K)]}{\rm CH}^1(F)^0\subseteq {\rm Ann}_{\Z[G(F/K)]}  Cl(\mathcal O_F),$$
for all proper $O_K$--ideals $\bf b$ coprime to ${\bf f'}\cdot w_F$. Indeed, for a fixed ideal ${\bf b}$, this is a direct
consequence of the previous remark applied to $\Theta_0(\frak p, {\bf f'})$, for all primes $\frak p\mid{\bf b}$ and the obvious equality
\begin{equation}\label{theta-product}\Theta_0({\bf ad, f'})={\rm N}{\bf d}\cdot\sigma_{\bf d}^{-1}\cdot\Theta_0({\bf a, f'}) + \Theta_0({\bf d, f'}),
\end{equation}
which holds for all proper $O_K$--ideals ${\bf a}$ and ${\bf d}$ coprime to ${\bf f'}\cdot w_F$. Moreover, by combining equalities \eqref{theta-product} and
\eqref{divisors-prime} above, one can easily show that if ${\bf b}$ is a proper $\CO_K$--ideal, coprime to ${\bf f'}\cdot w_F$ and ${\bf c}$ is any fractional $\CO_F$--ideal, then
there exists a unique element ${\wlambda}_{{\bf b, f'}}({\bf c})\in F^\times/\mu_F$, such that
\begin{equation}\label{divisor-general}{\rm div}_F(\wlambda_{{\bf b, f'}}({\bf c}))=\Theta_0({\bf b, f'})\cdot\widetilde{\bf c}.\end{equation}
Indeed, if ${\bf c}$ is coprime to ${\bf b}$, this is obvious. For an arbitrary ${\bf c}$,
we pick a fractional $\CO_F$--ideal ${\bf d}$ coprime to ${\bf b}$ and an $\widehat\varepsilon\in F^\times/\mu_F$, such that we have an
equality
$$\widetilde{\bf c}=\widetilde{\bf d} +{\rm div}_F(\widehat\varepsilon)+ D_\infty,$$
in ${\rm Div}_{S_\infty}(F)$, where $D_\infty\in\bigoplus_{w\in S_\infty}\R\cdot w.$ We hit the above equality with $\Theta_0({\bf b, f'})$ to obtain
$$\Theta_0{(\bf b, f'})\cdot\widetilde{\bf c}={\rm div}_F(\wlambda_{\bf b, f'}({\bf d})\cdot \widehat{\varepsilon}^{\Theta_0({\bf b, f'})}).$$
Therefore, we can set $\wlambda_{\bf b, f'}({\bf c}):=\wlambda_{\bf b, f'}({\bf d})\cdot \widehat{\varepsilon}^{\Theta_0({\bf b, f'})}.$

The reader
should notice  that if {\it ${\bf b}$ is not prime,} then, in general, it is not true that $\Theta_{0}({\bf b}, {\bf f'})$ annihilates
the ray--class group $Cl(\mathcal O_F)_{\bf b}$  of conductor ${\bf b}\CO_F$ (i.e. the elements $\wlambda_{{\bf b, f'}}({\bf c})$ cannot be chosen
to be classes in $F^\times/\mu_F$ of elements in $F^\times$ which are congruent to  $1\mod {\bf b}\CO_F$ even if ${\bf c}$ is coprime to ${\bf b}$, in general.)

\end{remark}

\begin{definition}
Let $F/K$ and ${\bf f'}$ be as above. Assume that ${\rm BrSt}(F/K, S_{\bf f'})$ holds. Consider a pair $({\bf b, c})$ consisting of
a proper $\CO_K$--ideal ${\bf b}$ which is coprime to ${\bf f'}\cdot w_F$ and a fractional $\CO_F$--ideal ${\bf c}$.
Then, the unique element $\wlambda_{\bf b, f'}({\bf c})\in F^\times/\mu_F$ satisfying \eqref{divisor-general} is
called {\it the Brumer--Stark element} associated to $(F/K, {\bf f', b, c})$.
\end{definition}

\begin{lemma}\label{BS-elements-lemma} Under the assumptions and with the notations of the above definition, the following hold.
\begin{enumerate}
\item[(i)] For all pairs $({\bf b, c})$ as above, we have an equality in $F^\times$
$$\lambda^{(1+j)}=1,$$
for any $\lambda$ in $F^\times$ whose class in $F^\times/\mu_F$ is $\wlambda_{\bf b, f'}({\bf c}).$
\item[(ii)] For all pairs $({\bf a, c})$ and $({\bf b, c})$ as above, we have an equality in $F^\times/\mu_F$
$$\wlambda_{\bf ab, f'}({\bf c})=\wlambda_{\bf a, f'}({\bf c})^{{\rm N}{\bf b}\cdot\sigma_{\bf b}^{-1}}\cdot \wlambda_{\bf b, f'}({\bf c}).$$

\end{enumerate}
\end{lemma}
\begin{proof} In order to prove (i), one combines \eqref{divisor-general} and \eqref{odd-theta}, to conclude that
$${\rm div}_F(\lambda^{(1+j)})=(1+j)\cdot\Theta_0({\rm b, f')}\cdot\widetilde{\bf c}=0.$$
Consequently, $\lambda^{(1+j)}\in\mu_F\cap F^+=\{\pm 1\}$, where $F^+=F^{j=1}$ is the maximal totally real subfield of $F$.
However, since $F$ is CM, $x^{(1+j)}$ is a totally positive element of $F^+$, for all $x\in F^\times$. This implies that $\lambda^{(1+j)}=1$, as stated.

The equality in (ii) is a direct consequence of \eqref{theta-product} and \eqref{divisor-general}.
\end{proof}

The considerations in the last two remarks lead naturally to the following.

\begin{proposition}\label{strong-BS-elements} Fix $F/K$ and ${\bf f'}$ as above. Assume that conjecture ${\rm BrSt}(F/K, S_{\bf f'})$ holds true. Fix a fractional $\CO_F$--ideal ${\bf c}$
and let $\mathcal B_{\bf c, f'}(F/K)$ be the set of all proper $\CO_K$--ideals $\bf b$ which are
coprime to ${\bf cf'}\cdot w_F$.
Then, there exist unique elements $\lambda_{\bf b, f'}({\bf c})$ in $F^\times$, for all ${\bf b}\in\mathcal B_{\bf c, f'}(F/K)$,
such that the following are satisfied.
\begin{enumerate}
\item[(i)] If ${\bf b}\in \mathcal B_{\bf c, f'}(F/K)$, then ${\rm div}_F(\lambda_{\bf b, f'}({\bf c}))=\Theta_{0}({\bf b, f'})\cdot\widetilde{\bf c}.$
\item[(ii)] If ${\bf b}$ is a prime ideal in $\mathcal B_{\bf c, f'}(F/K)$, then $\lambda_{\bf b, f'}({\bf c})\in F_{\bf b}^\times.$
\item[(iii)] If ${\bf a, b}\in \mathcal B_{\bf c, f'}(F/K)$,  then we have
$$\lambda_{\bf ab, f'}({\bf c})=\lambda_{\bf a, f'}({\bf c})^{{\rm N}{\bf b}\cdot\sigma_{\bf b}^{-1}}\cdot \lambda_{\bf b, f'}({\bf c}).$$
\end{enumerate}
Moreover, for all $\sigma\in G(F/K)$ and all ${\bf b}\in \mathcal B_{\bf c, f'}(F/K)(=\mathcal B_{\sigma({\bf c}), {\bf f'}}(F/K))$, we have
$$\lambda_{\bf b, f'}(\sigma({\bf c}))=\sigma(\lambda_{\bf b, f'}({\bf c})).$$
\end{proposition}
\begin{proof} Let us first note that {\it if ${\bf b}$ is prime,} then
conditions (i) and (ii) determine $\lambda_{\bf b, f'}({\bf c})\in F_{\bf b}^\times$ uniquely. Indeed, Remark \ref{annihilation-prime}
shows that $\lambda_{\bf b, f'}({\bf c})$ is the unique element in $F_{\bf b}^\times$, such that
$$\Theta_{0}({\bf b}, {\bf f'})\cdot \widetilde{\bf c}={\rm div}_F(\lambda_{{\bf b}, \bf f'}(\bf c)).$$
We define elements $\lambda_{\bf b, f'}(\bf c)$ satisfying (i) and (ii) for all ${\bf b}\in \mathcal B_{\bf c, f'}(F/K)$ by induction on the number of (not necessarily distinct)
prime factors of ${\bf b}$ as follows. For ideals ${\bf b}$ with one prime factor this has just been achieved. Now, assume that we have achieved this
for ideals with $(n-1)$ prime factors, for some $n\geq 2$. Let ${\bf b'}\in \mathcal B_{\bf c, f'}(F/K)$ equal to a product of $n$ primes. Let ${\frak p}$ be a prime dividing
$\bf b'$ and let ${\bf b'}={\bf b}\frak p$. Obviously, we have $\mathfrak p, {\bf b}\in \mathcal B_{\bf c, f'}(F/K)$. We let
\begin{equation}\label{induction}
\lambda_{{\bf b'}, {\bf f'}}(\bf c):=\lambda_{{\bf b}, {\bf f'}}(\bf c)^{{\rm N}\frak p\cdot\sigma_{\frak p}^{-1}}\cdot \lambda_{\frak p, {\bf f'}}(\bf c),\end{equation}
which clearly satisfies (i), by Lemma \ref{BS-elements-lemma}(ii).
In order to check that this definition does not depend on our choice of the prime ${\frak p}$, we need to show that
\begin{equation}\label{commuting}\lambda_{{\bf b}, {\bf f'}}(\bf c)^{{\rm N}\frak p\cdot\sigma_{\frak p}^{-1}}\cdot \lambda_{\frak p, {\bf f'}}(\bf c)=\lambda_{{\bf b}, {\bf f'}}(\bf c)\cdot \lambda_{\frak p, {\bf f'}}(\bf c)^{{\rm N}{\bf b}\cdot\sigma_{\bf b}^{-1}}.\end{equation}
However, equation \eqref{theta-product} shows that both sides of the equality to be proved have the same (Arakelov) divisor, namely ${\Theta_0({\bf b}\frak p, {\bf f'})}\cdot\widetilde{\bf c}$. This means
that the two sides differ by a root of unity, say $\zeta\in\mu_F$. This implies that
$$\zeta= \lambda_{{\bf b}, {\bf f'}}(\bf c)^{{\rm N}\frak p\cdot\sigma_{\frak p}^{-1}-1}\cdot \lambda_{\frak p, {\bf f'}}(\bf c)^{1-{\rm N}{\bf b}\cdot\sigma_{\bf b}^{-1}}.$$
Now, since $\lambda_{{\bf b}, {\bf f'}}(\bf c)$ is coprime to $\frak p$ and $\lambda_{\frak p, {\bf f'}}({\bf c})\in F_{\frak p}^\times$, the last equality implies that
$\zeta\in F_{\frak p}^\times$. However, since $\frak p$ is coprime to $w_F$, this implies that $\zeta=1$, which concludes the proof of \eqref{commuting}.

Now, based on the inductive construction \eqref{induction} given above, it is easily proved that (iii) is satisfied. Also, the uniqueness of $\{\lambda_{{\bf b, f'}}({\bf c})\mid {\bf b}\in\mathcal B_{\bf c, f'}(F/K)\}$ follows
immediately from (iii) and the uniqueness of $\lambda_{{\bf b, f'}}({\bf c})$ for ${\bf b}$ prime in $\mathcal B_{\bf c, f'}(F/K)\}$.

The last statement in the Proposition follows by checking first that the set $\{\sigma(\lambda_{{\bf b'}, {\bf f'}}(\bf c))\mid {\bf b}\in\mathcal B_{\bf c, f'}(F/K)\}$ satisfies
properties (i)--(iii) with ${\bf c}$ replaced with $\sigma({\bf c})$. This follows from the obvious fact that the map ${\rm div}_F$ is $G(F/K)$--equivariant. Finally, one uses uniqueness
in order to prove the last equality in the Proposition.
\end{proof}

\begin{definition} Let $F/K$ and ${\bf f'}$ be as above. Assume that ${\rm BrSt}(F/K, S_{\bf f'})$ holds.
Then, for any fractional $\CO_F$--ideal ${\bf c}$ and any proper $\CO_K$--ideal ${\bf b}$ corime to ${\bf cf'}\cdot w_F$, the unique element
$\lambda_{\bf b, f'}({\bf c})$ in $F^\times$ produced by Proposition \ref{strong-BS-elements} is called {\it the strong Brumer--Stark element} associated to the data
$(F/K, {\bf f'}, {\bf b, c})$.
\end{definition}

\begin{remark}  The astute reader would have noticed, no doubt, that throughout the current section the elements
$\widehat\lambda_{\bf b, f'}(\bf c)$ and $\lambda_{\bf b, f'}(\bf c)$ depend only on ${\rm Supp}(\bf f')$ and not on
${\bf f'}$ per se. This is a direct consequence of Remark \ref{support}.
\label{support of f prime}\end{remark}

\begin{lemma}\label{norm-lemma} Under the hypotheses of Proposition \ref{strong-BS-elements}, assume in addition that
$E/K$ is an abelian CM extension of $K$, such that $F\subseteq E$, and ${\bf e'}$ is a $\CO_K$--ideal divisible by the primes dividing the conductor
of $E/K$ and those dividing ${\bf f'}$. Also, assume that conjectures ${\rm BrSt}(E/K, S_{\bf e'})$ and ${\rm BrSt}(F/K, S_{\bf f'})$ hold. Let ${\bf c}$ be a fractional $\CO_{E}$--ideal. Then, for any proper $\CO_K$--ideal ${\bf b}$ coprime to ${\bf ce'}\cdot w_{E}$, we have
$$N_{E/F}(\lambda_{\bf b, e'}({\bf c})) = \lambda_{\bf b, f'}(N_{E/F}({\bf c}))^{\prod_{{\frak p\mid {\bf e'}}\atop{{\frak p}\nmid {\bf f'}}}(1-\sigma_{\frak p}^{-1})},$$
where $N_{E/F}$ denotes the usual norm from $E$ down to $F$ at the level of elements and fractional ideals
in $E$, and the product is taken over prime $\CO_K$--ideals $\frak p$.
\end{lemma}
\begin{proof} The proof is straightforward. The main idea is to prove that the elements
$\{N_{E/F}(\lambda_{\bf b, e'}({\bf c}))\mid {\bf b}\}$ and $\{\lambda_{\bf b, f'}(N_{E/F}({\bf c}))^{\Pi_{\frak p}(1-\sigma_{\frak p}^{-1})}\mid {\bf b}\}$ satisfy properties (i)--(iii) in the statement of
Proposition \ref{strong-BS-elements} with ${\bf c}$ and ${\bf f'}$ replaced by $N_{E/F}({\bf c})$ and ${\bf e'}$, respectively, and then use the uniqueness
property. Checking (ii) and (iii) is immediate. In order to check (i), first one uses the inflation property of
Artin $L$--functions to prove that
$${\rm Res}_{E/K}(\Theta^E_0({\bf e', b}))= \Theta_0({\bf e', b})= \Theta_0({\bf f', b})\cdot \prod_{{\frak p\mid {\bf e'}}\atop{{\frak p}\nmid {\bf f'}}}(1-\sigma_{\frak p}^{-1}),$$
for all ${\bf b}$ as above, where $\Theta^E_0({\bf e', b})$ is the Stickelberger element in $\Z[G(E/K)]$ associated to the data
$(E/K, {\bf e', b})$
and ${\rm Res}_{E/K}:\Z[G(E/K)]\to \Z[G(F/K)]$ is the Galois restriction group ring morphism. Secondly, one uses
the easily verified
$$Tr_{E/F}\circ{\rm div}_{E}={\rm div}_F\circ N_{E/F},$$
where $Tr_{E/F}: {\rm Div}_{S_\infty}(E)\to {\rm Div}_{S_\infty}(F)$ is the usual trace at the level of Arakelov divisors.
We leave the details to the interested reader.
\end{proof}

\begin{proposition}[Hecke characters]\label{Hecke-characters} Let $F/K$ and ${\bf f'}$ be as above. Assume that ${\rm BrSt}(F/K, {\bf f'})$ holds and fix
a proper $\CO_K$--ideal ${\bf b}$, coprime to ${\bf f'}\cdot w_F$. Let $I_F({\bf b})$ denote the group of fractional $\CO_F$--ideals coprime to ${\bf b}$.
Then, the following hold:
\begin{enumerate}
\item[(i)] For all ${\bf c, c'}\in I_F({\bf b})$, we have
$$\lambda_{\bf b, f'}({\bf c\cdot c'})=\lambda_{\bf b, f'}({\bf c})\cdot \lambda_{\bf b, f'}({\bf c'}).$$
\item[(ii)] If $\varepsilon\in F^\times$, such that $\varepsilon\equiv 1\mod^{\times}{\bf b}\CO_F$, we have
$$\lambda_{\bf b, f'}(\varepsilon O_F)=\varepsilon^{\Theta_0({\bf b, f'})}.$$
\item[(iii)] The group morphism
$$\lambda_{\bf b, f'}: I_F({\bf b})\to F^\times, \qquad {\bf c}\to \lambda_{\bf b, f'}({\bf c})$$
is a Hecke character for $F$ of conductor ${\bf b}\CO_F$ and of infinite type $\Theta_0({\bf b, f'})$.
\item [(iv)] The Hecke character $\lambda_{\bf b, f'}$ is $G(F/K)$--equivariant.
\end{enumerate}
\end{proposition}
\begin{proof} In order to prove (i), fix ${\bf c}$ and ${\bf c'}$ as above. Then observe that both sets
$$\{\lambda_{\bf a, f'}(\bf cc')\mid {\bf a}\in\mathcal B_{\bf cc', f'}(F/K)\}, \qquad  \{\lambda_{\bf a, f'}(\bf c)\cdot\lambda_{\bf a, f'}(\bf c') \mid {\bf a}\in\mathcal B_{\bf cc', f'}(F/K)\}$$
satisfy properties (i--iii) in Proposition \ref{strong-BS-elements} for the fractional $\CO_F$--ideal ${\bf c\cdot c'}$.
Then, apply the uniqueness property of these elements.

In order to prove (ii), observe that since $\varepsilon^{\Theta_0({\frak p}, {\bf f'})}\equiv 1\mod^{\times}{\frak p}\CO_F$ for any prime ideal $\frak p\mid{\bf b}$
and ${\rm div}_F(\varepsilon^{\Theta_0({\frak p}, {\bf f'})})=\Theta_0({\frak p}, {\bf f'})\cdot\widetilde{\varepsilon\CO_F}$ (since ${\rm div}_F$ is $G(F/K)$--equivariant),
we have an equality in $F^\times$
$$\lambda_{{\frak p}, {\bf f'}}(\varepsilon O_F)=\varepsilon^{\Theta_0(\frak p, {\bf f'})},$$
for any such prime $\frak p$. Now, (ii) follows from the above equality and Proposition \ref{strong-BS-elements}(iii), by induction
on the number of prime factors of ${\bf b}$.

The statement in (iii) is a direct consequence of (i) and (ii) and the definition of a Hecke character of a given conductor ${\bf b}$ and given infinite type  $\Theta\in \Z[G(F/K)]$.

Finally, (iv) follows from the last statement of Proposition \ref{strong-BS-elements}.
\end{proof}

\begin{remark} The Hecke characters $\lambda_{\bf b, f'}$ constructed above are vast generalizations
of Weil's Jacobi sum Hecke characters (\cite{Weil}). Weil's construction can be obtained from the above
when setting $K=\Q$. Note that in that case the Brumer--Stark conjecture is known to hold due
to Stickelberger's classical theorem (see \cite{Popescu-PCMI}, Section 4.3 for more details.)
The effort to ``align general Brumer--Stark elements'' into Hecke characters was initiated by
David Hayes in \cite{Hayes} and achieved with different methods and at a lower level of generality
by Yang in \cite{Yang}. However, the reader should be aware of the fact that the proof of the $2$--primary part of Theorem 5 in \cite{Yang} is incorrect.
\end{remark}

\begin{lemma}\label{extend-lemma} Under the assumptions of Proposition \ref{Hecke-characters}, there exists a (not necessarily unique) $\Z[G(F/K)]$--module morphism
$$\lambda^\ast_{\bf b, f'}: I_F\to F^\times$$
which extends $\lambda_{\bf b, f'}$ and satisfies the properties
\begin{equation}\label{extend}{\rm div}_F(\lambda^\ast_{\bf b, f'}({\bf c}))=\Theta_0({\bf b, f'})\cdot\widetilde{\bf c}, \qquad  \lambda^\ast_{\bf b, f'}({\bf c})^{(1+j)}=1,
\end{equation}
for all ${\bf c}\in I_F.$
\end{lemma}
\begin{proof} For ideals ${\bf c}\in I_F({\bf b})$, we set $\lambda^\ast_{\bf b, f'}({\bf c}):=\lambda_{\bf b, f'}({\bf c}).$
Now, we need to define $\lambda^\ast_{\bf b, f'}(w)$ for $\CO_F$--primes $w$ which divide ${\bf b}$.
Since we obviously have
$$(1-N{\bf b}\cdot \sigma_{\bf b}^{-1})\in {\rm Ann}_{\Z[G]}(\mu_F),$$
Lemme 1.1 in \cite[p. 82]{Tate} implies that we can write (not uniquely)
$$(1-N{\bf b}\cdot \sigma_{\bf b}^{-1})=\sum_{i=1}^n x_i\cdot(1-N{\bf p}_i\cdot \sigma_{{\bf p}_i}^{-1}),$$
for some $n\in\Bbb N$, some $x_1,\dots,x_n\in\Z[G]$, and some $\CO_K$--primes ${\bf p}_1, \dots, {\bf p}_n$ which are coprime to ${\bf bf'}\cdot w_F$. Let us fix
$n$, $x_1, \dots, x_n$ and ${\bf p}_1, \dots, {\bf p}_n$ with the above properties. Note that we have
$$\Theta_0({\bf b, f'})=\sum_{i=1}^n x_i\cdot \Theta_0({\bf p}_i, {\bf f'}).$$
Now, for each $\CO_F$--prime $w$ with $w\mid {\bf b}$ (therefore $w\in I_F({\bf p}_i)$, for all $i$), we define
$$\lambda^\ast_{\bf b, f'}(w):=\prod_{i=1}^n \lambda_{{\bf p}_i,{\bf f'}}(w)^{x_i}.$$
This extends $\lambda^\ast_{\bf b, f'}$ to all $\CO_F$--primes dividing ${\bf b}$.
Finally, for any ${\bf c}\in I_F$ we set
$$\lambda^\ast_{\bf b, f'}({\bf c}):=\prod_{\frak p\mid {\bf c}}\lambda^\ast_{\bf b, f'}(\frak p)^{n_{\frak p}},\qquad \text{  if  } {\bf c}=\prod_{\frak p}{\frak p}^{n_{\frak p}}.$$
Above, the product is taken over all the $\CO_F$--primes $\frak p$. The reader can easily check that the map $\lambda^\ast_{\bf b, f'}$ satisfies all the desired properties.
\end{proof}

\begin{remark}\label{norm-relations-star}
If ${\bf c}=\epsilon\CO_F$ is a principal ideal generated by $\epsilon\in F^\times$, then
$$\lambda^\ast_{\bf b, f'}(\epsilon\CO_F)=\xi\cdot\epsilon^{\Theta_0({\bf b, f'})},$$
for some root of unity $\xi\in\mu_F$, as the divisor equality \eqref{extend} easily implies.

Also, with notations as in Lemma \ref{norm-lemma}, once we pick extensions $\lambda^\ast_{\bf b, e'}$ and $\lambda^\ast_{\bf b, f'}$, the norm relations between $\lambda^\ast_{\bf b, e'}({\bf c})$ and $\lambda^\ast_{\bf b, f'}(N_{E/F}({\bf c}))$ become
$$N_{E/F}(\lambda^\ast_{\bf b, e'}({\bf c})) = \xi_{\bf c}\cdot\lambda^\ast_{\bf b, f'}(N_{E/F}({\bf c}))^{\prod_{{\frak p\mid {\bf e'}}\atop{{\frak p}\nmid {\bf f'}}}(1-\sigma_{\frak p}^{-1})},$$
where $\xi_{\bf c}$ is a root of unity in $\mu_F$, which depends on ${\bf c}$. This follows easily from the divisor equality \eqref{extend}. Of course, if ${\bf b}$ and ${\bf c}$ are coprime,
then $\xi_{\bf c}=1$.
\end{remark}

\subsection{The imprimitive $l$--adic theory.} Very recently, Greither and the second author (see \cite{GP}, section 6.1) proved a strong form of the imprimitive
Brumer-Stark conjecture, away from its $2$--primary part and
under the hypothesis that certain Iwasawa $\mu$--invariants vanish (as conjectured
by Iwasawa.) In what follows, we will state a weak form of the main result in loc.cit. This result turns
out to imply the existence of an imprimitive $l$--adic version of (strong) Brumer--Stark elements and Hecke characters, for all
odd primes $l$, which is sufficient for the $K$--theoretic constructions which follow.

\begin{theorem}[Greither-Popescu, \cite{GP}]\label{GP}
Let $F/K$ be as above. Let $l$ be an odd prime and assume that the Iwasawa $\mu$--invariant $\mu_{F,l}$ associated
to $F$ and $l$ vanishes. Assume that ${\bf f'}$ is a proper $\CO_K$--ideal divisible by all the primes dividing
${\bf f}l$. Then,  for all prime $\CO_K$--ideals ${\bf b}$ coprime to ${\bf f'}\cdot w_F$, we have
$$\Theta_0({\bf b}, {\bf f'})\in {\rm Ann}_{\zl[G(F/K)]} ({\rm CH}^1(F)^0_{T_{\bf b}}\otimes\zl).$$
\end{theorem}

\begin{remark} Recall that a major conjecture in number theory due to Iwasawa states that $\mu_{F,l}=0$, for all
primes $l$. At this point, this conjecture is only known to hold if $F$ is an abelian extension of $\Q$. The general belief is that it holds in general.

For a given odd prime $l$, the above theorem only settles an $l$--imprimitive form of $l$--adic piece of the Brumer-Stark conjecture for $F/K$.
That is so because the ideal ${\bf f'}$ is forced to be divisible by all $l$--adic primes, whether these ramify in $F/K$ or not. Consequently, $\Theta_0({\bf b}, {\bf f'})$
is obtained by multiplying $\Theta_0({\bf b}, {\bf f})$ with a $\zl[G]$--multiple of the element
$$\Pi'_{{\bf l}\mid l}(1-\sigma_{\bf l}^{-1}),$$
(product taken over the $l$--adic primes ${\bf l}$ in $K$ which do not divide ${\bf f}$)
which is not invertible in $\zl[G]$, in general, therefore leading to a weaker annihilation result. On the other hand, for any $n\in\Z_{\geq 1}$,
$\Theta_n({\bf b}, {\bf f'})$
is obtained by multiplying $\Theta_n({\bf b}, {\bf f})$ with a $\zl[G]$--multiple of the element
$$\Pi'_{{\bf l}\mid l}(1-\sigma_{\bf l}^{-1}\cdot N{\bf l}^n),$$ which is invertible in $\zl[G]$. This explains why imprimitivity
is not an issue in our upcoming $K$--theoretic considerations.

At this point, only very partial results towards the $2$--primary piece of the Brumer-Stark conjecture have been proved, which is
the reason why we are staying away from $l=2$ throughout the rest of this paper.
\end{remark}

For a given odd prime $l$, we extend the divisor map ${\rm div}_{F}$ by $\zl$--linearity to
$${\rm div}_{F}\otimes{\mathbf 1}_{\zl}: F^\times\otimes\zl \to {\rm Div}_{S_\infty}^0(F)\otimes\zl.$$
However, for the sake of simplicity, we use ${\rm div}_{F}$ to denote this extension as well, whenever
the prime $l$ has been chosen and fixed.
The following consequences of Theorem \ref{GP} have identical proofs to those
of Lemma \ref{BS-elements-lemma}, Proposition \ref{strong-BS-elements} and Proposition \ref{Hecke-characters}, respectively.

\begin{corollary}[imprimitive $l$--adic Brumer-Stark elements]\label{GP-BS-elements} Assume that the hypotheses of Theorem \ref{GP} hold. Let ${\bf b}$ be a proper
$\CO_K$--ideal coprime to ${\bf f'}\cdot w_F$ and let ${\bf c}$ be a fractional $\CO_F$--ideal. Then, the following hold.
\begin{enumerate}
\item [(i)] If ${\bf b}$ is a prime not dividing ${\bf c}$,
then there exists a unique element $\lambda_{{\bf b, f'}}({\bf c})\in F^\times_{\bf b}\otimes\zl$, such that
$${\rm div}_F(\lambda_{{\bf b, f'}}({\bf c}))=\Theta_0({\bf b, f'})\cdot\widetilde{\bf c}.$$
\item[(ii)] There exists a unique element $\wlambda_{{\bf b, f'}}({\bf c})\in (F^\times/\mu_F)\otimes\zl$, such that
$${\rm div}_F(\wlambda_{{\bf b, f'}}({\bf c}))=\Theta_0({\bf b, f'})\cdot\widetilde{\bf c}.$$
Further, any $\lambda\in F^\times\otimes\zl$ whose class in $(F^\times/\mu_F)\otimes\zl$ is $\wlambda_{{\bf b, f'}}({\bf c})$ satisfies
$$\lambda^{(1+j)}=1.$$
\end{enumerate}
\end{corollary}
\begin{proof} The proofs of (i) and the first equality in (ii) are identical to those of equalities \eqref{divisors-prime} and \eqref{divisor-general}, respectively.
The only difference is that instead of assuming the Brumer-Stark conjecture, here one uses Theorem \ref{GP}. Finally, the proof of the last equality
in (ii) is identical to that of Lemma \ref{BS-elements-lemma}, part (i).
\end{proof}

\begin{corollary}[imprimitive $l$--adic strong Brumer--Stark elements]\label{GP-strong-BS-elements} Assume that the hypotheses of Theorem \ref{GP} hold.
Let ${\bf c}$ be a fractional $\CO_F$--ideal. Let $\mathcal B_{\bf c, f'}(F/K)$ be the set of proper $\CO_K$--ideals ${\bf b}$ which are
coprime to ${\bf cf'}\cdot w_F$.
Then, there exist unique elements $\lambda_{\bf b, f'}({\bf c})$ in $F^\times\otimes\zl$, for all ${\bf b}\in\mathcal B_{\bf c, f'}(F/K)\otimes\zl$, such that:
\begin{enumerate}
\item[(i)] ${\rm div}_F(\lambda_{\bf b, f'}({\bf c}))=\Theta_{0}({\bf b, f'})\cdot\widetilde{\bf c}$, for all ${\bf b}\in \mathcal B_{\bf c, f'}(F/K)$.
\item[(ii)] If ${\bf b}$ is a prime ideal in ${\mathcal B}_{\bf c, f'}(F/K)$, then $\lambda_{\bf b, f'}({\bf c})\in F_{\bf b}^\times\otimes\zl.$
\item[(iii)] If ${\bf a, b}\in\mathcal B_{\bf c, f'}(F/K)\otimes\zl$,  then
$$\lambda_{\bf ab, f'}({\bf c})=\lambda_{\bf a, f'}({\bf c})^{{\rm N}{\bf b}\cdot\sigma_{\bf b}^{-1}}\cdot \lambda_{\bf b, f'}({\bf c}).$$
\end{enumerate}
 Moreover, for all $\sigma\in G(F/K)$ and all ${\bf b}\in \mathcal B_{\bf c, f'}(F/K)\otimes\zl(=\mathcal B_{\sigma({\bf c}\otimes\zl), {\bf f'}}(F/K))$, we have
$$\lambda_{\bf b, f'}(\sigma({\bf c}))=\sigma(\lambda_{\bf b, f'}({\bf c})).$$
\end{corollary}
\begin{proof}The proof  is identical to that of Proposition \ref{strong-BS-elements}.
\end{proof}

\begin{corollary}[imprimitive $l$--adic Hecke characters]\label{GP-Hecke-characters}
Assume that the hypotheses of Theorem \ref{GP} hold. Fix
a proper $\CO_K$--ideal ${\bf b}$, coprime to ${\bf f'}\cdot w_F$. Let $I_F({\bf b})$ denote the group of fractional $\CO_F$--ideals coprime to ${\bf b}$.
Then, we have:
\begin{enumerate}
\item[(i)] For all ${\bf c, c'}\in I_F({\bf b})\otimes\zl$,
$$\lambda_{\bf b, f'}({\bf c\cdot c'})=\lambda_{\bf b, f'}({\bf c})\cdot \lambda_{\bf b, f'}({\bf c'}).$$
\item[(ii)] If $\varepsilon\in F^\times$, such that $\varepsilon\equiv 1\mod^{\times}{\bf b}\CO_F$, we have
$$\lambda_{\bf b, f'}(\varepsilon O_F)=\varepsilon^{\Theta_0({\bf b, f'})}.$$
\item[(iii)] The $\Z_l$--module morphism
$$\lambda_{\bf b, f'}: I_F({\bf b})\otimes\zl\to F^\times\otimes\zl, \qquad {\bf c}\to \lambda_{\bf b, f'}({\bf c})$$
 is $G(F/K)$--invariant.
\end{enumerate}
\end{corollary}
\begin{proof} The proof is identical to that of Proposition \ref{Hecke-characters}(i), (ii), (iv).
\end{proof}

\begin{remark}\label{norm-extend-remark} Obvious analogues of Lemmas \ref{norm-lemma} (norm relations) and \ref{extend-lemma} (extension to
$I_F\otimes\zl$) hold for the
imprimitive $l$--adic maps $$\lambda_{\bf b, f'}:I_F({\bf b})\otimes\zl\to F^\times\otimes\zl$$ as well. We leave the details to the interested reader.
An extension of $\lambda_{\bf b, f'}$ to $I_F\otimes\zl$ as in Lemma \ref{extend-lemma} will be denoted by $\lambda^\ast_{\bf b, f'}$. Of course, the
imprimitive $l$--adic analogue of Remark \ref{norm-relations-star} holds.
\end{remark}

\bigskip

%-----\section{Galois equivariant Stickelberger splitting map}-----

\section{The Galois equivariant Stickelberger splitting map}

In this section, we will construct the $l$--adic Galois equivariant Stickelberger splitting
map in the Quillen localization sequence associated to the top field $F$ in an abelian Galois extension
$F/K$ of number fields, with $K$ totally real and $F$ either totally real or CM.
The main idea is to use the imprimitive $l$--adic Brumer--Stark elements for certain
cyclotomic extensions of $F$ along with powers of Bott elements to construct special elements in the $K$--theory of the top field $F$. Then, one
uses these special elements to construct the desired Galois equivariant splitting map.

 From now on, we fix an abelian extension $F/K$ as above, denote by $G$ its Galois group, fix an odd prime $l$  and work
under the assumption that Iwasawa's conjecture on the vanishing of the $\mu$--invariant associated to $F$
and $l$ holds. Further, we fix nontrivial $\CO_K$--ideals ${\bf f}$ and ${\bf b}$, with ${\bf f}$ divisible by the (finite) conductor
of $F/K$ and ${\bf b}$ coprime to $l{\bf f}$.

In the case $K=\Q$, a Stickelberger
splitting map was constructed in \cite{Ba1}. The construction in loc. cit. was refined in \cite{BG1}, \cite{BG2}. However, none of
these constructions led to Galois equivariant splitting maps.

In \cite{BP}, we constructed a Galois equivariant Stickelberger
splitting map for arbitrary totally real base fields $K$. However, that construction was very different from the one
we are about to describe in that it relies on a different class of special elements in the $K$--groups of the top field $F$.
\medskip

\subsection{$K$--theoretic tools.} In what follows, we will use freely $K$--theory with(out) coefficients as well as
the theory of Bockstein morphisms and that of Bott elements at the level of $K$--theory with coefficients. For the precise
definitions and main properties the reader can consult \S3 in \cite{BP}. However, just to set the notations, we will briefly
recall the main objects and facts.

Let $R$ be a unital ring and $l$ be an odd prime number. Then the $K$--groups with coefficients $K_n(R, \Z/l^k)$, $n\geq 1$, $k\geq 1$,  sit inside short
exact sequences

\begin{equation}\label{Bockstein-sequence}\xymatrix{
0\ar[r] &K_n(R)/l^k\ar[r] &K_n(R, \Z/l^k)\ar[r]^{b=b_R} &K_{n-1}(R)[l^k]\ar[r] &0\,,}
\end{equation}
where $b$ (sometimes denoted $b_R$, to emphasize dependence on the ring) is the Bockstein morphism associated to $R$, $n$ and $l^k$. We remind the reader that once $l^k$ and $n$ are fixed, $b$ and the exact sequences above
are functorial in $R$.

If  we assume that the characteristic of $R$ is not $l$ and that $R$ contains the group $\mu_{l^k}$ of $l^k$--roots of unity and fix a generator $\xi_{l^k}$ of  $\mu_{l^k}$,
then we have canonical special elements $\beta(\xi_{l^k})$ in $K_2(R, \Z/l^k)$ called Bott elements. Consequently, the product structure ``$\ast$'' at the level of $K$--theory with coefficients leads to canonical  elements $\beta(\xi_{l^k})^{\ast n}$
in $K_{2n}(R, \Z/l^k)$, for all $n\geq 1$. For given $l^k$ and $n$, the elements $\beta(\xi_{l^k})^{\ast n}$ are functorial in $R$ and the chosen $\xi_{l^k}$ in the obvious sense.
\medskip

\subsection{Constructing maps $\Lambda_v$.} In this section, we fix an integer $n\geq 1$, an odd prime $l$, and a nonzero $\CO_F$--prime $v$. We let $k_v:=\CO_F/v$ denote the residue field of $v$.
Our main goal is to construct a group morphism
$$\Lambda_v: K_{2n-1}(k_v)_l\to K_{2n}(F)_l$$
satisfying certain properties. Recall that the group $K_{2n-1}(k_v)$ is cyclic of order $q_v^n-1$, where $q_v=\mid k_v\mid$. The idea behind constructing
$\Lambda_v$ is first to get our hands on an explicit generator $\xi_v$ of $K_{2n-1}(k_v)_l$ and then construct an explicit element in $K_{2n}(F)_l$ annihilated by the order of this generator
and declare that to be the image of $\xi_v$ via $\Lambda_v$
\medskip

Obviously, if $\mid K_{2n-1}(k_v)_l\mid=1$, then $\xi_v=1$ and $\Lambda_v$ is the trivial map. So, let us assume that
 $\mid K_{2n-1}(k_v)_l\mid=l^k$, for some $k>0$. This implies that $v\nmid l$. Also, it is easily seen (see the proof
of Lemma 2 \cite[p. 336]{Ba1})
that this also implies that
$$
k>v_l(n),
$$
where $v_l(n)$ denotes the usual $l$--adic valuation of $n$.

Next, we let $E:=F(\mu_{l^k})$ and fix an $\CO_E$--prime $w$ sitting above $v$. It is easily seen (for full proofs see the proof
of Lemma 2 \cite[p. 336]{Ba1}) that $k_w=k_v(\mu_{l^k})$ and consequently that
\begin{equation}\label{order of K_{2n-1}(k_w)}\mid K_{2n-1}(k_w)_l\mid =l^{v_l(n)+k}\end{equation}
and that the image of the transfer map $Tr_{w/v}:K_{2n-1}(k_w)_l\to K_{2n-1}(k_v)_l$ satisfies
\begin{equation}\label{image-transfer}{\rm Im}(Tr_{w/v})=K_{2n-1}(k_v)_l.\end{equation}
Fix a generator $\xi_{l^k}$ of $\mu_{l^k}$ inside $E$. By abuse of notation, we will denote by $\xi_{l^k}$ its image in $k_w$ via
the reduction modulo $w$ map $\CO_E\to \CO_E/w=k_w$. This way, we obtain a generator $\xi_{l^k}$ of $\mu_{l^k}$ inside $k_w$. A result of Browder \cite{Br} shows
that $\beta(\xi_{l^k})^{\ast n}$ is a generator of $K_{2n}(k_w, \Z/l^k)$. On the other hand, since $K_{2n}(k_w)=0$, the Bockstein sequence \eqref{Bockstein-sequence} gives a group isomorphism
$$b: K_{2n}(k_w, \Z/l^k)\simeq K_{2n-1}(k_w)[l^k].$$
Therefore, $b(\beta(\xi_{l^k})^{\ast n})$ is a generator of $K_{2n-1}(k_w)[l^k]$. Consequently, equality \eqref{image-transfer} allows us to make the following.

\begin{definition}\label{picking xi_v} Let $\xi_v$ be a generator
of $K_{2n-1}(k_v)_l$, such that
\begin{equation}\label{define-xiv}
\xi_v^{l^{v_l(n)}}=Tr_{w/v}(b(\beta(\xi_{l^k})^{\ast n}).
\end{equation}
\end{definition}
\medskip

Now, we proceed to the construction of a special element in $K_{2n}(F)_l$ annihilated by $l^k$.
Let $\Gamma:=G(E/K)$ and let $G_v$ and $I_v$ denote the decomposition and
inertia groups of $v$ in $G(F/K)$, respectively. Also, let
$${\bf f}_E^\ast:={\bf f}\cdot l.$$
Note that ${\bf f}_E^\ast$ is divisible by all the primes which ramify in $E/K$ and all the $l$--adic primes and that ${\bf b}$ and ${\bf f}_E^\ast$ are coprime. If we denote
by $\{\Theta^E_m({\bf b}, {\bf f}^\ast_E)\}_{m\geq 0}$ the higher Stickelberger elements associated
to the data $(E/K, {\bf b}, {\bf f}_E^\ast)$, then these are all in $\zl[\Gamma]$. Further, if we set
\begin{equation}\label{gamma_l}\gamma_l:=\prod_{{{\bf l}\mid l}\atop{{\bf l}\nmid{\bf f}}}(1-\sigma_{\bf l}^{-1}\cdot N{\bf l}^n)^{-1} \in \zl[G],\end{equation}
then we have the obvious equality
\begin{equation}\label{restriction-theta}{\rm Res}_{E/F}(\Theta^E_n({\bf b}, {\bf f}^\ast_E)=\Theta_n({\bf b}, {\bf f})\cdot\gamma_l^{-1}.\end{equation}

Also, note that $E/K$ is
a CM abelian extension of a totally real number field. Consequently, Lemma \ref{GP-Hecke-characters} applies to the data $(E/K, {\bf f}_E^\ast, {\bf b})$. In particular, Remark \ref{norm-extend-remark} allows us to pick a $\zl[\Gamma]$--linear morphism
$$\lambda_{{\bf b}, {\bf f}_E^\ast}^\ast: I_E\otimes\zl\to E^\times\otimes\zl,$$
which extends the $l$--adic imprimitive Hecke character  $\lambda_{{\bf b}, {\bf f}_E^\ast}$ of conductor ${\bf b}$
and satisfies the properties
\begin{equation}\label{extension}{\rm div}_E(\lambda^\ast_{{\bf b}, {\bf f}_E^\ast}({\bf c}))=\Theta^E_0({\bf b}, {\bf f}^\ast_E)\cdot{\bf c}, \qquad
\lambda^\ast_{{\bf b}, {\bf f}^\ast_E}({\bf c})^{(1+j)}=1,
\end{equation}
for all ${\bf c}\in I_E$.

Let us pick an $\CO_E$--prime $w$ sitting above $v$. We let $\Gamma_w$ and $I_w$ denote its decomposition and inertia groups in $E/K$. We view  $\lambda^\ast_{{\bf b}, {\bf f}_E^\ast}(w)$ as an element
in $K_1(E)_l$ after the obvious identification $E^\times\otimes\zl\simeq K_1(E)_l$. Consequently, we obtain an element
$$\lambda^\ast_{{\bf b}, {\bf f}_E^\ast}(w)\ast b(\beta(\xi_{l^k})^{\ast n})\in K_1(E)_l\ast K_{2n-1}(E)[l^k]\subseteq K_{2n}(E)[l^k]$$
which is mapped via the usual transfer morphism $Tr_{E/F}: K_{2n}(E)_l\to K_{2n}(F)_l$ to
$$Tr_{E/F}(\lambda^\ast_{{\bf b}, {\bf f}_E^\ast}(w)\ast b(\beta(\xi_{l^k})^{\ast n}))\in K_{2n}(F)[l^k].$$

\begin{definition}\label{local-lambda-definition}Since the chosen generator $\xi_v$ of $K_{2n-1}(k_v)_l$ has order $l^k$, there exists
a unique $\zl$--linear map
$\Lambda_v:K_{2n-1}(k_v)_l\to K_{2n}(F)_l$
which satisfies
$$\Lambda_v(\xi_v):=Tr_{E/F}(\lambda^\ast_{{\bf b}, {\bf f}_E^\ast}(w)\ast b(\beta(\xi_{l^k})^{\ast n}))^{\gamma_l
}.$$
\end{definition}

\begin{remark} Note that the map $\Lambda_v$ depends in an easily described manner on the several choices we have made along the way: that of a prime $w$ sitting above $v$ in
$E$, that of a generator $\xi_{l^k}$ of $\mu_{l^k}$ in $E$, that of a generator $\xi_v$ of $K(k_v)_l$ and, finally, that of a $\zl[\Gamma]$--linear extension $\lambda^\ast_{{\bf b}, {\bf f}_E^\ast}$ of the $l$--adic imprimitive Hecke character  $\lambda_{{\bf b}, {\bf f}_E^\ast}$ of conductor ${\bf b}$.
\end{remark}
\medskip

The functoriality properties of $K$--groups imply that we have the following obvious isomorphisms of $\zl[G]$-- and $\zl[\Gamma]$--modules, respectively, for all $m\geq 0$.
\begin{equation}\label{identify-F} K_{m}(k_v)_l\otimes_{\zl[G_v]}\zl[G]\simeq \bigoplus_{\widehat\sigma\in G/G_v}K_{m}(k_{\sigma(v)})_l,\quad  \xi\otimes\sigma\to (1,\dots,1, {\sigma(\xi)},1, \dots, 1)\end{equation}
\begin{equation}\label{identify-K} K_{m}(k_w)_l\otimes_{\zl[\Gamma_w]}\zl[\Gamma]\simeq \bigoplus_{\widehat\gamma\in \Gamma/\Gamma_w}K_{m}(k_{\gamma(w)})_l, \quad \xi\otimes\gamma\to (1,\dots,1, {\gamma(\xi)},1, \dots, 1).\end{equation}
Above $\sigma\in G$, $\gamma\in\Gamma$, $\widehat\sigma$ and $\widehat\gamma$ are their classes in $G/G_v$ and $\Gamma/\Gamma_w$, respectively, and $\sigma(\xi)$ and $\gamma(\xi)$ appear in the $\sigma(v)$ and $\gamma(w)$--components, respectively.
In what follows, we will freely identify the left and right hand sides of these isomorphisms. In particular, if $\xi\in K_{m}(k_w)_l$ (or $\xi\in K_{m}(k_v)_l$) and
$\alpha\in\zl[\Gamma]$ (or $\alpha\in\zl[G]$), then $\xi\otimes\alpha$ will also be sometimes denoted by $\xi^\alpha$ and thought of as an element
in the direct sum on the right hand side of isomorphisms \eqref{identify-F} and \eqref{identify-K}.

\begin{remark} Let us note that if we set ${\bf c}:=w$, then \eqref{extension} can be rewritten as
\begin{equation}\label{boundary-lambda}
\partial_E(\lambda^\ast_{{\bf b}, {\bf f}_E^\ast}(w))=\partial_{E,\Gamma\cdot w}(\lambda^\ast_{{\bf b}, {\bf f}_E^\ast}(w))=1\otimes \Theta^E_0({\bf b}, {\bf f}^\ast_E),
\end{equation}
where $\partial_{E, \Gamma\cdot w}:K_{1}(E)_l\to \bigoplus_{\widehat\gamma\in\Gamma/\Gamma_w} K_{0}(k_{\gamma(w)})_l$ is the $\Gamma\cdot w$--supported boundary map and we identify
$$\bigoplus_{\widehat\gamma\in\Gamma/\Gamma_w} K_{0}(k_{\gamma(w)})_l\simeq K_0(k_w)_l\otimes_{\zl[\Gamma_w]}\zl[\Gamma]\simeq\zl\otimes_{\zl[\Gamma_w]}\zl[\Gamma]$$
as in the last displayed isomorphism above for $m=0$.
Also, let us note that since $K_{2n-1}(k_v)_l$ is a cyclic group (also cyclic $\zl[G_v]$--module) generated by $\xi_v$, we have
$$\bigoplus_{\widehat\sigma\in G/G_v}K_{2n-1}(k_{\sigma(v)})_l\simeq K_{2n-1}(k_v)_l\otimes_{\zl[G_v]}\zl[G]=\zl[G]\cdot(\xi_v\otimes 1).$$
Therefore, $\bigoplus_{\widehat\sigma\in G/G_v}K_{2n-1}(k_{\sigma(v)})_l$  is a cyclic $\zl[G]$--module generated by $\xi_v$.
\end{remark}

\begin{theorem}\label{local-lambda-theorem}
The map $\Lambda_v$ defined above satisfies the following properties.
\begin{enumerate}
\item It is $\zl[G_v]$--linear.
\item If $\partial_{F, G\cdot v}:K_{2n}(F)_l\to \bigoplus_{\widehat\sigma\in G/G_v} K_{2n-1}(k_{\sigma(v)})_l$ is the $G\cdot v$--supported boundary map then,
after the identification \eqref{identify-F} for $m:=2n-1$, we have
$$\partial_F(\Lambda_v(\xi))=\partial_{F,G\cdot v}(\Lambda_v(\xi))=\xi^{{l^{v_l(n)}\cdot\Theta_n({\bf b}, {\bf f})}},$$
for all $\xi\in K_{2n-1}(k_v)_l$.
\end{enumerate}
\end{theorem}
\begin{proof}
(1) Let $\sigma\in G_v$. Assume that $\sigma=\sigma_v^\alpha\cdot\rho$, where $\sigma_v$ is a $v$--Frobenius lift (from $G_v/I_v$ to $G_v$),
$\alpha\in\Z_{\geq 0}$ and $\rho\in I_v$. Since $\Lambda_v$ is $\zl$--linear, we have
$$\Lambda_v(\sigma(\xi_v))=\Lambda_v(\sigma_v^\alpha(\xi_v))=\Lambda_v(\xi_v^{q_v^{n\alpha}})=\Lambda_v(\xi_v)^{q_v^{n\alpha}}.$$
Now, let us consider a lift $\overline\sigma\in\Gamma_w$ of $\sigma$ of the form $\overline\sigma=\sigma_w^\alpha\cdot\overline\rho$,
where $\sigma_w$ is a $w$--Frobenius lift in $\Gamma_w$ which restricts to $\sigma_v$ and $\overline\rho\in I_w$ restricts to $\rho$. We have
$$\overline\sigma(\lambda^\ast_{{\bf b}, {\bf f}_E^\ast}(w))=\lambda^\ast_{{\bf b}, {\bf f}_E^\ast}(\overline\sigma(w))=\lambda^\ast_{{\bf b}, {\bf f}_E^\ast}(w).$$
Also, since $v\nmid l$, we have $\overline{\sigma}(\xi_{l^k})=\sigma_w^\alpha(\xi_{l^k})=\xi_{l^k}^{q_v^\alpha}$. Consequently, the functoriality of the Bockstein
and Bott maps $b$ and $\beta$, respectively, gives
$$\overline\sigma(b(\beta(\xi_{l^k})^{\ast n}))=b(\beta(\xi_{l^k})^{\ast n})^{q_v^{n\alpha}}.$$
The last two displayed equalities imply that
\begin{eqnarray}
% \nonumber to remove numbering (before each equation)
\nonumber  {\sigma}(\Lambda_v(\xi_v)) &=& \sigma\circ Tr_{E/F}(\lambda^\ast_{{\bf b}, {\bf f}_E^\ast}(w)\ast b(\beta(\xi_{l^k})^{\ast n}))^{\gamma_l}= \\
\nonumber   Tr_{E/F}\circ\overline\sigma(\lambda^\ast_{{\bf b}, {\bf f}_E^\ast}(w)\ast b(\beta(\xi_{l^k})^{\ast n}))^{\gamma_l} &=& Tr_{E/F}(\overline\sigma(\lambda^\ast_{{\bf b}, {\bf f}_E^\ast}(w))\ast \overline\sigma(b(\beta(\xi_{l^k})^{\ast n})))^{\gamma_l}= \\
\nonumber   Tr_{E/F}(\lambda^\ast_{{\bf b}, {\bf f}_E^\ast}(w)\ast b(\beta(\xi_{l^k})^{\ast n}))^{q_v^{n\alpha}\gamma_l}&=&\Lambda_v(\xi_v)^{q_v^{n\alpha}}.
\end{eqnarray}
Consequently, $\Lambda_v(\sigma(\xi_v))=\sigma(\Lambda_v(\xi_v))$, which concludes the proof of (1).
\medskip

(2) Note that since $K_{2n-1}(k_v)_l$ is generated by $\xi_v$ and $\Lambda_v$ is $\zl[G_v]$--linear, it suffices to prove (2) for $\xi:=\xi_v$. For that purpose, we need the following commutative diagrams in the category of $\Z[\Gamma]$--modules. The first of these is diagram 4.7 in \cite{Ba1}:
\begin{equation}\label{diagram1}\xymatrix{
K_1(E)\times K_{2n-1}(O_E)\ar[d]^{\oplus_{\bf w}(\partial_{E,{\bf w}}\times{\rm red}_{\bf w})}\ar[r]^{\qquad\ast} &K_{2n}(O_E)\ar[d]^{\partial_E}\\
\bigoplus_{\bf w}(K_0(k_{\bf w})\times K_{2n-1}(k_{\bf w}))\ar[r]^{\qquad\ast} &\bigoplus_{\bf w} K_{2n-1}(k_{\bf w}).
}\end{equation}
 Above, $\bf w$ runs through all nonzero $\CO_E$--primes and ${\rm red}_{\bf w}: K_{2n-1}(O_E)\to K_{2n-1}(k_{\bf w})$ is the map at the level of $K$--groups induced
by the reduction modulo ${\bf w}$ morphism $O_E\to O_E/{\bf w}=k_{\bf w}$. In particular, note that the functoriality of the Bockstein
and Bott maps $b$ and $\beta$ implies that we have
\begin{equation}\label{reduce-Bockstein}{\rm red}_w(b_E(\beta(\xi_{l^k})^{\ast n}))= b_w(\beta(\xi_{l^k})^{\ast n}), \end{equation}
where $w$ is the chosen $\CO_E$--prime sitting above the chosen $\CO_F$--prime $v$.

The second commutative diagram is diagram 4.1 in \cite{Ba1}:
\begin{equation}\label{diagram2}\xymatrix{
K_{2n}(E)\ar[r]^{\partial_E\qquad}\ar[d]^{Tr_{E/F}} &\bigoplus_{\bf v}(\bigoplus_{{\bf w}|{\bf v}} K_{2n-1}(k_{\bf w}))\ar[d]^{\bigoplus_{\bf v}(\prod_{\bf w\mid v}Tr_{\bf w/v})=:\overline{Tr}}\\
K_{2n}(F)\ar[r]^{\partial_F\qquad} &\bigoplus_{\bf v} K_{2n-1}(k_{\bf v})}\end{equation}
where ${\bf v}$ runs through all nonzero $\CO_F$--primes.

Diagram \eqref{diagram2} above yields the equality
\begin{equation}\label{first-step}\partial_F(\Lambda_v(\xi_v))=\overline{Tr}(\partial_E(\lambda^\ast_{{\bf b}, {\bf f}_E^\ast}(w) \ast b_E(\beta(\xi_{l^k})^{\ast n})))^{\gamma_l}.\end{equation}
 Since $\mu_{l^k}\subseteq E^\times$, we have
$l^k\mid w_n(E)_l$. This observation combined with the functoriality of the maps $b$ and $\beta$ lead to the following equality in $\bigoplus_{\widetilde\gamma\in\Gamma/\Gamma_w}K_{2n-1}(k_{\gamma(w)})[l^k]$.
$$b_{\gamma (w)}(\beta(\xi_{l^k})^{\ast n})= b_{w}(\beta(\xi_{l^k})^{\ast n})^{t_n(\gamma)}, \qquad\text{ for all }\gamma\in\Gamma,$$
where $t_n:\zl/w_n(E)_l[\Gamma]\simeq\zl/w_n(E)_l[\Gamma]$ is the map defined right before Theorem \ref{Deligne-Ribet-congruences}.
Consequently,  \eqref{diagram1} above, combined with equalities \eqref{boundary-lambda} and \eqref{reduce-Bockstein} yield
$$\partial_E(\lambda^\ast_{{\bf b}, {\bf f}_E^\ast}(w)\ast b_E(\beta(\xi_{l^k})^{\ast n}))=\{b_w(\beta(\xi_{l^k})^{\ast n})\}^{t_n(\widehat{\Theta_0^E({\bf b}, {\bf f}_E^\ast)})},$$
where the notations are as in Theorem \ref{Deligne-Ribet-congruences}. Now, Theorem \ref{Deligne-Ribet-congruences} implies that we have
$$t_n(\widehat{\Theta_0^E({\bf b}, {\bf f}_E^\ast)})= \widehat{\Theta_n^E({\bf b}, {\bf f}_E^\ast)} \text{ in }\zl/w_n(E)_l[\Gamma].$$
Consequently, the last two displayed equalities imply that
$$\partial_E(\lambda^\ast_{{\bf b}, {\bf f}_E^\ast}(w)\ast b_E(\beta(\xi_{l^k})^{\ast n}))=\{b_w(\beta(\xi_{l^k})^{\ast n})\}^{\Theta_n^E({\bf b}, {\bf f}_E^\ast)}.$$
Now, combine the last equality successively with \eqref{first-step}, \eqref{define-xiv} and \eqref{restriction-theta} to obtain
$$\partial_F(\Lambda_v(\xi_v))=Tr_{w/v}(b_w(\beta(\xi_{l^k})^{\ast n}))^{\gamma_l\cdot\Theta_n({\bf b}, {\bf f})\cdot\gamma_l^{-1}}=\xi_v^{l^{v_l(n)}\cdot\Theta_n({\bf b}, {\bf f})}.$$
This concludes the proof of the theorem.
\end{proof}

\subsection{Constructing a map $\Lambda$.} For every nonzero $\CO_K$--prime $v_0$, we pick an $\CO_F$--prime $v\mid v_0$.
The isomorphisms \eqref{identify-F} for $m:=2n-1$ yield an explicit isomorphism of $\zl[G]$--modules
$$\bigoplus_{v'}K_{2n-1}(k_{v'})_l\simeq\bigoplus_{v_0}\left (\zl[G]\otimes_{\zl[G_v]}K_{2n-1}(k_v)_l\right ),$$
where $v'$ runs over all the nonzero $\CO_F$--primes and $v_0$ runs over all the nonzero $\CO_K$--primes.
For each chosen $v$, we construct a map $\Lambda_v: K_{2n-1}(k_v)_l\to K_{2n-1}(F)_l$ as in the previous section.
Since each of these maps $\Lambda_v$ is $\zl[G_v]$--linear (see Theorem \ref{local-lambda-theorem} (2)), the isomorphism
above implies that there exists a unique $\zl[G]$--linear map
$$\Lambda: \bigoplus_{v'}K_{2n-1}(k_{v'})_l\to K_{2n}(F)_l$$
which equals $\Lambda_v$ when restricted to $K_{2n-1}(k_v)_l$ for each of the chosen primes $v$.
Obviously, after identifying the two sides of the isomorphism above, this unique $\zl[G]$--linear map is given by
\begin{equation}\label{lambda-definition}\Lambda:=\prod_{v_0}(1\otimes\Lambda_v).\end{equation}
\begin{theorem}\label{lambda-properties} The map $\Lambda$ defined above satisfies the following.
\begin{enumerate}
\item It is $\zl[G]$--linear.
\item For all $\xi\in\bigoplus_{v'}K_{2n-1}(k_{v'})_l$, we have
$$(\partial_F\circ\Lambda)(\xi)=\xi^{l^{v_l(n)}\cdot\Theta_n({\bf b}, {\bf f})}.$$
\end{enumerate}
\end{theorem}
\begin{proof} Part (1) is a direct consequence of \eqref{lambda-definition} and Theorem \ref{local-lambda-theorem}(1).
Part (2) is a direct consequence of \eqref{lambda-definition} and Theorem \ref{local-lambda-theorem}(2).
\end{proof}

\section{The Stickelberger splitting map and the divisible elements}

In this section, we will use the special elements in the even $K$--groups of a totally real or CM number field constructed in the previous section
to investigate the Galois module structure of the groups of divisible elements in these $K$--groups. In particular, this will lead to
 a reformulation and a possible generalization of a classical conjecture of Iwasawa.
\medskip

The following technical lemma extends the computations in \cite[p. 8-9]{BG2} and will be needed shortly.

\begin{lemma}
Let $L/F$ be a Galois extension of number fields. Let $L^H$ be the Hilbert class field of $L$.
Let $l$ be a prime such that $L^{H} \cap L (\mu_{l^{\infty}}) = L.$
Let ${\bf p} \subset \mathcal{O}_L$ be a prime. For any positive integer $m>>0$
there exist infinitely many primes ${\bf q} \subset \mathcal{O}_L$, with ${\bf q}\nmid l$, such that:
\begin{enumerate}
\item $[{\bf q}] = [{\bf p}]$ in  $Cl (\mathcal{O}_L).$
\item $l^m\mid\mid N{\bf q}-1$.
\item ${\bf q}\cap O_F$ splits in $L/F$.
\end{enumerate}
\label{HilbertChebotarev}\end{lemma}
\begin{proof}
%$$\xymatrix{
%L (\mu_{l^{m}})  \quad \ar[r]^{\quad \quad }  &  \quad L (\mu_{l^{m}}) L^{H}    \\
%L \quad \ar@<0.1ex>[u]^{{\rm{Id}}}
%\ar[r]^{\qquad Fr_{{\bf p}}\qquad } &  \quad  L^{H}  \ar@<0.1ex>[u]^{}\,  }
%\label{HilbertMuInfty}$$
Let $m\in\Z_{>0}$ sufficiently large so that $L(\mu_{l^m})\ne L(\mu_{l^{m+1}})$. Since $L(\mu_{l^{m+1}})\cap L^H=L$, there exists a unique
$\sigma \in  G(L (\mu_{l^{m+1}}) L^{H}/F)$ such that $\sigma_{| L^{H}} = Fr_{{\bf p}}$,
$\sigma_{| L (\mu_{l^{m+1}})}$ is a generator of $G(L(\mu_{l^{m+1}})/L(\mu_{l^m}))$, where  $Fr_{{\bf p}}$ is the Frobenius automorphism associated to ${\bf p}$ in $G(L^H/L)$.
Note that, by definition $\sigma_{|L}={\rm Id}$ and therefore $\sigma\in G(G(L (\mu_{l^{m+1}}) L^{H}/L)$, which is an abelian Galois group.

By Chebotarev's density theorem, there are infinitely many
primes $q \subset \mathcal{O}_K$ such that ${q}\nmid l$ and $\sigma = \widetilde{Fr}_{q},$ where
$\widetilde{Fr}_{q}$ is (any) Frobenius automorphism associated to ${q}$ in $G(L(\mu_{l^{m+1}}) L^{H}/F).$
Let ${\bf q}$ be a prime in $O_L$ sitting above $q$.

Since $\sigma_{| L}={\rm Id}$, $q={\bf q}\cap O_F$ splits completely in $L/F$.

Since ${\widetilde{Fr}_{{\bf q}}} \, |_{L (\mu_{l^{m+1}})}$ is a generator of $G(L(\mu_{l^{m+1}})/L(\mu_{l^m}))$, we have
$$\widetilde{Fr_{\bf q}}(\xi)=\xi^{N{\bf q}}=1, \qquad \text{ for all }\xi\in\mu_{l^m},$$
and $\widetilde{Fr_{\bf q}}(\xi)=\xi^{N{\bf q}}\ne 1$ for a generator $\xi$ of $\mu_{l^{m+1}}$. Consequently, $l^m\mid\mid N{\bf q}-1$.

Finally, if we denote by $Fr_{{\bf q}}$ the Frobenius morphism associated to ${\bf q}$ in $G(L^H/L)$, we have $Fr_{{\bf q}}={\widetilde{Fr}_{{\bf q}}} \, |_{L^H} = Fr_{{\bf p}}$.
Consequently, Artin's reciprocity isomorphism
$$Cl (\mathcal{O}_L) \rightarrow G(L^H/L),\qquad [{\bf a}] \rightarrow Fr_{\bf a},$$
implies that $[{\bf p}]=[{\bf q}]$ in $Cl (\mathcal{O}_L)$.
\end{proof}
\bigskip

We work with the notations and under the assumptions of the previous section. In addition, we will
assume from now on that the odd prime $l$ does not divide the order $|G|$ of the Galois group $G:=G(F/K)$. We denote by $\widehat G(\overline\Q_l)$ the set
of irreducible $\overline\Q_l$--valued characters of $G$. For $\chi\in \widehat G(\overline\Q_l)$ we let $$e_\chi:=1/|G|\sum_{g\in G}\chi(g)\cdot g^{-1}$$ denote its
associated idempotent element in $\zl[\chi][G]$. Also, we will let
$$\widetilde{e_\chi}:=\sum_{\sigma\in G(\overline\Q_l/\Q_l)}e_{\chi^\sigma}.$$
Note that $\widetilde{e_\chi}$ is the irreducible idempotent in $\zl[G]$ associated to the irreducible $\Q_l$--valued
character $\widetilde\chi=\sum_{\sigma\in G(\overline\Q_l/\Q_l)}{\chi^\sigma}$ of $G$. Also, note that $\widetilde{e_\chi}$ only depends on
the orbit of $\chi$ under the natural action of $G({\Q_l}/\Q_l)$ on $\widehat G(\overline \Q_l)$. In what follows, we denote the set of such orbits
by $\widehat G(\Q_l)$ and think of $\widetilde\chi$ as an element of $\widehat G(\Q_l)$, for every $\chi\in\widehat G(\overline\Q_l)$. Obviously, we have
$$\zl[G]=\bigoplus_{\widetilde\chi\in\widehat G(\Q_l)}\widetilde{e_\chi}\zl[G], \qquad \widetilde{e_\chi}\zl[G]\simeq \zl[\chi],$$
where the ring isomorphism above sends $x\to\chi(x)$, for every $x\in \widetilde{e_\chi}\zl[G]$ and $\chi\in\widehat G(\overline\Q_l)$. Also, for every $\zl[G]$--module
$M$ we have
\begin{equation}M=\bigoplus_{\widetilde\chi\in\widehat G(\Q_l)}\widetilde{e_\chi}M,\label{decompose into chi components}\end{equation}
where $\widetilde{e_\chi}M$ is a $\widetilde{e_\chi}\zl[G]$--module in the obvious manner. From now on, we denote
$$M^\chi:=\widetilde{e_\chi}M$$
and view it as a $\zl[\chi]$--module via the ring isomorphism  $\widetilde{e_\chi}\zl[G]\simeq \zl[\chi]$ described above.
Obviously, $M\to M^\chi$ are exact functors from the category of $\zl[G]$--modules to that of $\zl[\chi]$--modules.
If $f:M\to N$ is a morphism of $\zl[G]$--modules, then $f^\chi: M^\chi\to N^\chi$ denotes its image
via the above functor. Also, if $M$ is as above and $x\in M$, then
$x^\chi:=\widetilde{e_\chi}\cdot x$ will be viewed as an element in $M^\chi$. In particular, if $x\in\zl[G]$, then we identify $x^\chi$ with $
\chi(x)\in\zl[\chi]$ via the ring isomorphism $\widetilde{e_\chi}\zl[G]\simeq \zl[\chi]$ described above.
\medskip

From now on, we fix an embedding $\Bbb C\hookrightarrow \Bbb C_l$. As a consequence, the higher Stickelberger elements
$\Theta_{n} ({\bf b}, {\bf f})$ will be viewed in $\zl[G]\subseteq \C_l[G]$. Also, this embedding
identifies $\widehat G(\C)$ and $\widehat G(\overline\Q_l)$. Therefore, Remark \ref{L-function-theta} and the conventions made above give
the following equality in $\zl[\chi]$, for all $\chi\in\widehat G(\overline\Q_l)$:
\begin{equation}\label{expressing Stickeleberger via L functions 3}\Theta_{n} ({\bf b}, {\bf f})^\chi=(1 - N {\bf b}^{n+1}\cdot{\chi}(\sigma_{\bf b})^{-1})\cdot L_{\bf f}({\chi}^{-1}, -n).
\end{equation}
\medskip

Let us fix $\chi\in\widehat G(\overline\Q_l)$ and $n\geq 1$. Consider the $\chi$ component of the Quillen localization sequence
(\ref{Quillen localization sequence}):
\begin{equation}
0 \rightarrow K_{2n} (\mathcal{O}_F)_{l}^{\chi} \rightarrow
K_{2n} (F)_{l}^{\chi} \stackrel{\partial_{F}^{\chi}}{\longrightarrow}
\bigoplus_{v_0} \bigl(\bigoplus_{v | v_0} K_{2n-1} (k_v)_l\bigr)^{\chi} \rightarrow 0.
\label{chi part of the Quillen localization sequence}\end{equation}
Observe that since $\bigoplus_{v | v_0} K_{2n-1} (k_v)_l$ is a cyclic
$\Z_l[G]$-module generated by $\xi_v$ (see \eqref{identify-F} in the previous section), $\bigl(\bigoplus_{v | v_0} K_{2n-1} (k_v)_l\bigr)^\chi$
is a cyclic $\zl[\chi]$--module generated by $\xi_v^\chi$. It is easily seen that since $K_{2n-1} (k_v)_l\simeq \Z/l^k$ (with notations as in the previous section), isomorphisms \eqref{identify-F} imply that
\begin{equation}\label{order of xivchi} {\rm ord}(\xi_v^\chi)=l^k, \text{ whenever }\xi_v^\chi\ne 1.\end{equation}
Now, since the map $\Lambda$ constructed in the previous section
is $\zl[G]$--equivariant, Theorem \ref{lambda-properties} implies that we have
\begin{equation}
\partial_{F}^{\chi} \circ \Lambda^{\chi}  \, (\xi) =
\xi^{l^{v_{l} (n)} \Theta_{n} ({\bf b}, {\bf f})^{\chi}},\qquad\text{ for all }\xi\in \bigoplus_{v_0}\bigl(\bigoplus_{v | v_0} K_{2n-1} (k_v)_l\bigr)^\chi.
\label{chi part theoremStickelbergerElementsProperty1}
\end{equation}
\medskip

\begin{definition} If $A$ is an abelian group, we denote by $div(A)$ its subgroup of divisible elements.
In other words, we let
$$div(A)=\bigcap_{r\geq 1} A^r.$$
\end{definition}

Let $D (n) := div (K_{2n} (F))$. It is easy to see that we have
\begin{equation} D(n)_l^\chi=div(K_{2n}(F)_l^\chi),\label{chi part of div is div of chi part}\end{equation}
for all $\chi$ as above. Also, observe that (\ref{chi part of the Quillen localization sequence}) combined with the finiteness
of the Quillen $K$--groups of finite fields implies right away that for all $\chi$ and $n$ we have
\begin{equation}
D (n)^\chi_{l} \subset K_{2n} (\mathcal{O}_F)_{l}^{\chi}.
\label{Dnchi subset k2nOF}
\end{equation}
\medskip

\noindent {\bf Simplifying Hypothesis:} {\it From now on we assume that all primes above $l$ are ramified in $F(\mu_l)/K$ and totally
ramified in $F(\mu_{l^\infty})/F(\mu_l).$}
\medskip

\begin{theorem}[]\label{divisible elements via Lambda}
Let $\chi \in \widehat G(\C)$ and assume that $v_l(n)=0$ and
$\Theta_{n} ({\bf b}, {\bf f})^{\chi} \not= 0.$ Then
\begin{equation}
K_{2n} (\mathcal{O}_F)_{l}^{\chi} \cap \,\, {\rm{Im}}(\Lambda^{\chi}) = D (n)^\chi_{l}.
\label{divisible elements via Lambda 1}
\end{equation}
In particular, if  $v_l(n)=0$ and $\Theta_{n} ({\bf b}, {\bf f})^{\chi} \not= 0$
for all $\chi \in \widehat G(\C)$, then
\begin{equation}
K_{2n} (\mathcal{O}_F)_{l}  \cap \,\, {\rm{Im}}(\Lambda) = D (n)_{l}.
\label{divisible elements via Lambda 3}
\end{equation}
\end{theorem}
\begin{proof} Note that the exactness of the functors $M\to M^\chi$ combined with \eqref{decompose into chi components} shows
that equality (\ref{divisible elements via Lambda 3}) follows from equalities (\ref{divisible elements via Lambda 1}), for all $\chi\in\widehat G(\C)$.
So, we proceed with the proof of (\ref{divisible elements via Lambda 1}).
\medskip

First, take $d \in D (n)^{\chi}_{l}$ and take a natural number $m$ such that
$$l^m > l^{v_{l} (\Theta_{n} ({\bf b}, {\bf f})^{\chi})} |K_{2n} (\mathcal{O}_F)_{l}^{\chi}|.$$
Write $d$ as $d = x^{l^m}$, for some $x \in K_{2n} (F)_{l}^{\chi}.$
By \eqref{chi part theoremStickelbergerElementsProperty1}, we have
\begin{equation}
\partial_{F}^{\chi}\circ \Lambda^{\chi}\circ \partial_{F}^{\chi} (x) =
\partial_{F}^{\chi} (x)^{ \Theta_{n} ({\bf b}, {\bf f})^{\chi}}
\end{equation}
Hence, we have
\begin{equation}
\Lambda^{\chi} \circ \partial_{F}^{\chi} (x) = x^{ \Theta_{n} ({\bf b}, {\bf f})^{\chi}} y
\label{LambdaPartial1}
\end{equation}
for an element $y \in K_{2n} (\mathcal{O}_F)_{l}^{\chi}.$
Raising (\ref{LambdaPartial1}) to the power $D_{\chi} := l^{m}\cdot |\Theta_{n} ({\bf b}, {\bf f})^{\chi}|_{l}^{-1}$ gives:
\begin{equation}
\Lambda^{\chi} \circ \partial_{F}^{\chi} (x^{D_{\chi}}) = x^{l^m} = d.
\label{LambdaPartial111}
\end{equation}
Hence, we have  $d \in K_{2n} (\mathcal{O}_{F})_{l}^{\chi} \cap \,\, {\rm{Im}}\,\, \Lambda^{\chi}.$
\medskip

Now, assume that $y \in K_{2n} (\mathcal{O}_{F})_{l}^{\chi} \cap \,\, {\rm{Im}}\,\,
\Lambda^{\chi}.$ For each prime $v_0 \nmid l$ of $\mathcal{O}_K$ fix a prime
$v | v_0$ in $\CO_F$. With notations as in the previous section, we can write
\begin{equation}
y = \prod_{v} \, \Lambda^{\chi} (\xi_{v}^{\chi})^{c_v},
\label{Product for y}\end{equation}
where $c_v \in  \Z_l[\chi].$
Since $\partial_{K}^{\chi} (y) = 1$,  \eqref{chi part theoremStickelbergerElementsProperty1} implies that
$$(\xi_v^\chi)^{c_v  \Theta_{n} ({\bf b}, {\bf f})^{\chi}}=1, $$
for each of the chosen primes $v$. This, combined with \eqref{order of xivchi} implies that
\begin{equation}\label{divisible by l^k} l^k\mid c_v  \Theta_{n} ({\bf b}, {\bf f})^{\chi}, \text{ whenever }\xi_v^\chi\ne 1.
\end{equation}

Now, let us fix one of the chosen primes $v$ and assume that $\xi_v^\chi\ne 1$. With notations as in the previous section, \eqref{lambda-definition} and Theorem \ref{lambda-properties}(1) imply that
\begin{equation}
\Lambda^{\chi} (\xi_{v}^{\chi}) = \widetilde{e_{\chi}}\cdot Tr_{E/F} (\lambda^\ast_{{\bf b}, {\bf f}^\ast_E} (w) \ast
b(\beta(\xi_{l^k})^{\ast\, n}))^{\gamma_l}.
\label{description of Lambda chi xi chi}
\end{equation}

Let $m\in\Z_{\geq k}$ be any integer such that
\begin{equation}\label{m large} l^{m-k}\,\geq\, \mid K_{2n}(O_F)_l\mid\end{equation}
Since all the $l$--adic primes are
totally ramified in the extension $F(\mu_{l^{\infty}}) / F(\mu_l)$,  Lemma \ref{HilbertChebotarev} applied to the extension $E/F$
allows us to choose a prime $w_2$ of
$E$ which is coprime to ${\bf b}l$ such that: $v_2:=w_2\cap O_F$ splits completely in $E/F$, and $[w] = [w_2] \in Cl(\mathcal{O}_E)$, and $l^m\mid\mid (N w_2 - 1)$
for $m$ large.  Since $v_l(n)=0$ and $v_2$ splits completely in $E/F$, we have $l^m=|K_{2n-1}(k_{v_2})_l|$.
Let $\widetilde{w_2}$ be a prime of $E(\mu_{l^m})$ over $w_2.$
Under our simplifying hypothesis, the projection formula (see \cite{We}, Chapter V, \S3.3.2)
combined with Lemma \ref{norm-lemma}, Remark \ref{support of f prime} and Lemma \ref{HilbertChebotarev}(2)
gives the following relation:

\begin{equation}
\Lambda (\xi_{v_2})^{l^{m-k}} = Tr_{E(\mu_{l^m})/F} (\lambda^\ast_{{\bf b}, {\bf f}_{E}}
(\widetilde{w_2}) \ast b(\beta(\xi_{l^m})^{\ast\, n}))^{\gamma_l\cdot l^{m-k}} =
\nonumber\end{equation}
\begin{equation}
= Tr_{E(\mu_{l^m})/F} (\lambda^\ast_{{\bf b}, {\bf f}_{E}} (\widetilde{w_2}) \ast
b(\beta(\xi_{l^k})^{\ast\, n}))^{\gamma_l} =
\nonumber\end{equation}
\begin{equation}
= Tr_{E/F} (\lambda^\ast_{{\bf b}, {\bf f}_{E}} (w_2) \ast
b(\beta(\xi_{l^k})^{\ast\, n}))^{\gamma_l}
\label{RelationForLambdas}\end{equation}
By our choice of $w_2$ and Corollary \ref{GP-BS-elements}, we have

\begin{equation}
\lambda^\ast_{{\bf b}, {\bf f}_{E}} (w_2) = \lambda^\ast_{{\bf b}, {\bf f}_{E}} (w) \alpha^{\Theta_{0} ({\bf b}, {\bf f}_{E})} u
\label{comp. of lambda for w and w2}\end{equation}
for some $\alpha \in E^\times\otimes\zl$ and $u \in \mu_E\otimes\zl.$
Since $l$ is odd, $u = N_{E(\mu_{l^m})/E} (u_2)$
for some $u_2 \in \mu_{E(\mu_{l^m})}\otimes\zl$. Consequently, by the projection formula and \eqref{m large},
\begin{equation}
Tr_{E/F} (u \ast
b(\beta(\xi_{l^k})^{\ast\, n}))^{\gamma_l} =
Tr_{E(\mu_{l^m})/F} (u_2 \ast
b(\beta(\xi_{l^m})^{\ast\, n}))^{\gamma_l})^{l^{m-k}}=1.
\label{tr of unit twisted with Bott is divisible}
\end{equation}
Now, the projection formula combined with Theorem \ref{Deligne-Ribet-congruences} gives
\begin{equation}
b (Tr_{E/F} (\alpha^{\Theta_{0} ({\bf b}, {\bf f}_{E})} \ast
\beta(\xi_{l^k})^{\ast\, n})^{\gamma_l})^{c_v} =
{b (Tr_{E/F} (\alpha \ast
\beta(\xi_{l^k})^{\ast\, n}))}^{c_v \, \Theta_{n} ({\bf b}, {\bf f})}
\label{alpha to theta 0 case}\end{equation}

\noindent
Hence, if we multiply (\ref{alpha to theta 0 case}) by $\widetilde{e_\chi}$ and apply \eqref{divisible by l^k}, we obtain:

\begin{equation}
\widetilde{e_{\chi}} b (Tr_{E/F} (\alpha^{\Theta_{0} ({\bf b}, {\bf f}_{E})} \ast
\beta(\xi_{l^k})^{\ast\, n})^{\gamma_l})^{c_v} =
\widetilde{e_{\chi}} {b (Tr_{E/F} (\alpha \ast
\beta(\xi_{l^k})^{\ast\, n}))}^{c_v \, \Theta_{n} ({\bf b}, {\bf f})^{\chi}} = 1.
\label{alpha to theta 0 case chi part}\end{equation}
 Now
(\ref{RelationForLambdas}), (\ref{comp. of lambda for w and w2}),
(\ref{tr of unit twisted with Bott is divisible}) and
(\ref{alpha to theta 0 case chi part}) imply that:
\begin{equation}
\Lambda^{\chi} (\xi_{v}^{\chi})^{c_v} = \Lambda^{\chi} (\xi_{v_2}^{\chi})^{c_{v}l^{m-k}}.
\label{divisibility by l m - k}\end{equation}
The product in (\ref{Product for y}) ranges in fact over a finite set of primes $v$ that depends on $y$
and $m$ is arbitrarily large. Hence (\ref{Product for y}) and (\ref{divisibility by l m - k}) show that
$d \in D(n)^\chi_l.$
\end{proof}
\medskip

If we restrict Theorem \ref{divisible elements via Lambda} to the situation $K=F$, we obtain the following
condition for the cyclicity of the group of divisible elements $D(n)_l$ in $K_{2n} (K)_l$, for an arbitrary totally
real number field $K,$ under certain hypotheses.

\begin{theorem} \label{generalize Iwasawa} Let $K$ be a totally real number field,
$l$ an odd prime, and $n\geq 1$ an odd integer such that $v_l(n)=0$.
Let $D(n)_l$ be the group of divisible elements in
$K_{2n}(K)_l.$ Assume that $K$ satisfies the simplifying hypothesis above and that
$|\prod_{v | l} w_{n} (K_v)|_l = 1.$
Then, the following conditions are equivalent:
\begin{itemize}
\item[{(1)}] The group $D(n)_l$ is cyclic.
\item[{(2)}] There exists a prime ideal $v_0$ and an $\CO_K$--ideal ${\bf b}$ coprime to $w_{n+1}(K)_l$ such that $| K_{2n-1} (k_{v_0})_l|$
is divisible by $|(1 - N {\bf b}^{n+1}) \zeta_{K} (-n)|_{l}^{-1}$  and the map
$$
\Lambda_{v_0} \, :\, K_{2n-1}(k_{v_0})_l \rightarrow K_{2n} (K)_l
$$
associated to the data $(K/K, n, l, {\bf b}, v_0)$ is injective.
\end{itemize}
\end{theorem}
\begin{proof}
We start with a few preliminary remarks. Let ${\bf b}$ be an $\CO_K$--ideal coprime to $w_{n+1}(K)$. Note that since $K$ is totally real and $n$ is odd,
we have
$$(1 - N {\bf b}^{n+1}) \zeta_{K} (-n)\ne 0.$$
(See \cite{K}, p. 198.) Consequently, Theorem \ref {divisible elements via Lambda} applies to the data $(K/K, n, \chi={1})$. Construct a map $\Lambda$ for the data $(K/K, n, l, {\bf b})$ as in the previous section.
By Theorem \ref{divisible elements via Lambda} (see (\ref{divisible elements via Lambda 3})) and \eqref{divisible by l^k} we can write
every $y \in D(n)_l$ as
\begin{equation}
y = \prod_{v} \, \Lambda (\xi_{v})^{c_v},
\label{Product for y third time}
\end{equation}
where $c_v \in  \Z_l$ are such that
\begin{equation}\label{order}{\rm ord}(\xi_v)\mid c_v (1 - N {\bf b}^{n+1}) \zeta_{K} (-n),
\end{equation}
for each prime $v$ in $\mathcal{O}_K$. Hence, if we apply $\partial_K$ to
\eqref{Product for y third time}, we can conclude that
$\Lambda (\xi_{v})^{c_v} \in K_{2n} (\mathcal{O}_K)_l$ for each $v$.
So, again, by
Theorem \ref{divisible elements via Lambda} (equality (\ref{divisible elements via Lambda 3}))
we have
\begin{equation}\Lambda (\xi_{v})^{c_v} \in D (n)_l, \qquad\text{ for all }v.
\label{locally in dnl}\end{equation}
Now, \cite[Theorem 3 (ii)]{Ba2} and our assumption that  $|\prod_{v | l} w_{n} (K_v)|_l = 1$ imply that
\begin{equation}
|D(n)_l| = |w_{n+1} (K) \zeta_{K} (-n)|_{l}^{-1}.
\label{number of elements of D(n)l}
\end{equation}
Since $w_{n+1} (K) = {\rm gcd} \, (N {\bf b}^{n+1} - 1)$
where the ${\rm gcd}$ is taken over ideals ${\bf b}$ coprime with $w_{n+1} (K)$, there is an ideal ${\bf b},$
coprime with $w_{n+1} (K)_l,$ such that
\begin{equation}
|w_{n+1} (K) \zeta_{K} (-n)|_{l}^{-1} =
|(1-N {\bf b}^{n+1}) \zeta_{K} (-n)|_{l}^{-1}
\label{l parts of w zeta = l part of N b n+1 - 1}\end{equation}

{\bf (1) $\Rightarrow$ (2).} \,\, If $D(n)_l$  is trivial, then any $l$--adic prime $v_0$ satisfies the conditions in (2). Assume that $D(n)_l$ is cyclic and nontrivial. Let ${\bf b}$ be an $\CO_K$--ideal satisfying \eqref{l parts of w zeta = l part of N b n+1 - 1}. Construct a map $\Lambda$ for the data $(K/K, n, {\bf b})$.
 Relations (\ref{Product for y third time}) and \eqref{locally in dnl}
imply that there exists an $\CO_K$--prime $v_0$ such that $\Lambda_{v_0} (\xi_{v_0})^{c_{v_0}}$ is a generator of $D(n)_l$. (Take a $v_0$ such that $\Lambda_{v_0} (\xi_{v_0})^{c_{v_0}}$ has maximal order.)
By \eqref{number of elements of D(n)l} we have
$$|w_{n+1} (K)
\zeta_{K} (-n)|_{l}^{-1}=|(1 - N {\bf b}^{n+1}) \zeta_{K} (-n)|_{l}^{-1}={\rm ord}(\Lambda_{v_0} (\xi_{v_0})^{c_{v_0}})\leq {\rm ord}(\xi_{v_0}).$$
Since $\xi_{v_0}$ is a generator of $| K_{2n-1} (k_{v_0})_l|$, this implies that $$|(1 - N {\bf b}^{n+1}) \zeta_{K} (-n)|_{l}^{-1}\mid
| K_{2n-1} (k_{v_0})_l|.$$ Moreover, the map $\Lambda_{v_0}$
must be injective. Indeed, if ${\rm ord}(\xi_{v_0})=1$ this is clear. If ${\rm ord}(\xi_{v_0})>1$ and ${\rm ord}(\Lambda_{v_0} (\xi_{v_0}))<{\rm ord}(\xi_{v_0}),$ then
$1<{\rm ord}(\Lambda_{v_0} (\xi_{v_0})^{c_{v_0}})<{\rm ord}(\xi_{v_0}^{c_{v_0}}).$
But this is impossible, since $\Lambda_{v_0} (\xi_{v_0})^{c_{v_0}}$ has order $|w_{n+1} (K)
\zeta_{K} (-n)|_{l}^{-1}$ and this number also annihilates $\xi_{v_0}^{c_{v_0}}$ by \eqref{order}.
\medskip

{\bf (2) $\Rightarrow$ (1).} \,\, Let $r_{v_0}:=|K_{2n-1}(k_{v_0})_l |.$ Consequently,
the number
$$c_{v_0} := \frac{r_{v_0}}{w_{n+1} (K) \zeta_{K} (-n)}$$
has nonnegative $l$--adic valuation. Moreover, by Theorem \ref{local-lambda-theorem}
we have $$\partial_K(\Lambda_{v_0} (\xi_{v_0})^{c_{v_0}})=\xi_{v_0}^{r_{v_0}}=1.$$ Consequently,  $\Lambda (\xi_{v_0})^{c_{v_0}}\in K_{2n} (\mathcal{O}_K)_l.$
Hence by
Theorem \ref{divisible elements via Lambda} (equality (\ref{divisible elements via Lambda 3}))
we have $\Lambda_{v_0} (\xi_{v_0})^{c_{v_0}} \in D (n)_l.$ Since $\Lambda_{v_0}$ is injective,  we have
$${\rm ord}(\Lambda_{v_0} (\xi_{v_0})^{c_{v_0}})=|w_{n+1} (K) \zeta_{K} (-n)|_{l}^{-1}.$$ Now, \eqref{number of elements of D(n)l}
implies that $D (n)_l$ is cyclic generated by $\Lambda_{v_0} (\xi_{v_0})^{c_{v_0}}$.
\end{proof}

The following technical lemma refines the implication $(1)\Longrightarrow (2)$ in the above theorem and will be needed in the
next section.

\begin{lemma}\label{infinitely many v} Assume the hypotheses of Theorem \ref{generalize Iwasawa}, and assume that $D(n)_l$ is cyclic.
Then for any $\CO_K$--ideal ${\bf b}$ which is coprime to $w_{n+1}(K)_l$ and such that
\begin{equation}\label{choosing b}(1-N{\bf b}^{n+1})\Z_l=w_{n+1}(K)_l\Z_l\end{equation}
and any integer $m>>0$, there exist infinitely many $\CO_K$--primes $v$ such that
$\Lambda_v$ is injective and $l^m\mid (Nv-1)$.
\end{lemma}
\begin{proof} Let us fix a ${\bf b}$ which satisfies \eqref{choosing b}.
The proof of $(1)\Longrightarrow (2)$ in Theorem \ref{generalize Iwasawa} shows that if one has an $\CO_K$--ideal $v_0$ such that
$|(1 - N {\bf b}^{n+1}) \zeta_{K} (-n)|_l^{-1} \mid (Nv_0^{n}-1)$ and $\Lambda_{v_0}(\xi_{v_0})^{c_{v_0}}$ generates
$D(n)_l$ for some $c_{v_0}\in\Z$, then the map $\Lambda_{v_0}$ is injective. Let us fix a $v_0$ and a $c_{v_0}$ satisfying all these properties.

Let $m\in\Bbb Z_{\geq k}$ such that
$$l^{m-k}\geq \max \left\{ |(1 - N {\bf b}^{n+1}) \zeta_{K} (-n)|_l^{-1}, \quad |K_{2n}(\CO_K)_l|\right\}.$$
The technique developed in the proof of Theorem \ref{divisible elements via Lambda} allows us to construct infinitely many
$\CO_K$--primes $v$ which are coprime to ${\bf b}l$ and which satisfy
$$l^m\mid (Nv-1), \qquad \Lambda_v(\xi_v)^{c_{v_0}\cdot l^{m-k}}=\Lambda_{v_0}(\xi_{v_0})^{c_{v_0}}.$$
(See \eqref{divisibility by l m - k} and the arguments preceding it.) For any such prime $v$,  we have
$$|(1 - N {\bf b}^{n+1}) \zeta_{K} (-n)|_l^{-1} \mid (Nv^{n}-1)$$
and $\Lambda_v(\xi_v)^{c_{v_0}\cdot l^{m-k}}$ generates $D(n)_l$.
Consequently, the map $\Lambda_v$ is injective.
\end{proof}

\medskip

In the particular case $K = \Q$, the above theorem is closely related to a classical conjecture of
Iwasawa. We make this link explicit in what follows.
For that purpose, let $A$ be the $l$-torsion part of the class group of $\Z[\mu_l].$ Let
${\omega} \, :\, G(\Q(\mu_l)/\Q) \rightarrow \Z_{l}^{\times}$ be the Teichm{\" u}ller
character. For every $i\in\Z_{\geq 0}$, let $e_{\omega^i} \in
\Z_l [G(\Q(\mu_l) / \Q)]$ be the
idempotent associated to $\omega^i$. View $A$ as a $\zl[G(\Q(\mu_l)/\Q)]$--module and let
$$A^{[i]}:=A^{\omega^i}=e_{\omega^i} A,$$ in the
notations used at the beginning of this section.

\begin{conjecture} {\rm (Iwasawa)} \quad
$A^{[l-1 - n ]}$ is cyclic for all $n$ odd, $1 \leq n < l-1.$
\label{Iwasawa conjecture}\end{conjecture}

\noindent Recall the following result (see \cite{BG1}, \cite{BG2}):
\begin{proposition}\label{Iwasawa in terms of cyclicity} Let $n$ be odd, $1 \leq n < l-1.$ Then
$A^{[l-1 - n ]}$ is cyclic if and only if $D(n)_l$ is cyclic, where $D(n)_l$ is the group of divisible elements in
$K_{2n} (\Q)_l.$
\end{proposition}

In the next section, we will give a new proof of this result (see Theorem \ref{another look at the Iwasawa conj.}), based on our construction of the Galois equivariant Stickelberger splitting map $\Lambda$.
For the moment, let us simply observe that if combined with Proposition \ref{Iwasawa in terms of cyclicity}, Theorem \ref{generalize Iwasawa}
gives the following equivalent formulation of Iwasawa's conjecture in terms of our
map $\Lambda$.

\begin{corollary} \label{a different view at the Iwasawa conj.}
Let $n$ be an odd integer such that $1 \leq n < l-1.$
The following conditions are equivalent:
\begin{enumerate}
\item The group $A^{[l-1-n]}$ is cyclic.
\item The group $D(n)_l$ is cyclic.
\item There exists a prime number $p$ and an integer ${b} \in \Z_{\geq 1}$  coprime to $w_{n+1}(\Q)_l$ such that
$|(1 - {b}^{n+1}) \zeta_{\Q} (-n)|_{l}^{-1}$ divides  $| K_{2n-1} (\F_{p})_l|$ and the
map
$$
\Lambda_{p} \, :\, K_{2n-1}(\F_{p})_l \rightarrow K_{2n} (\Q)_l
$$
associated to the data $(\Q/\Q, n, l, {b}\Z, p)$  is injective.
\end{enumerate}
\end{corollary}
\begin{proof} This is a direct consequence of Theorem \ref{generalize Iwasawa} upon
observing that $v_{l} (n) = 0$ and $|w_{n} (\Q_l)|_l = 1$, for each $n$ odd and $1 \leq n < l-1$ and
also that $\Q$ satisfies our simplifying hypothesis for all odd primes $l$.
\end{proof}

%\begin{remark}\label{splitting remark} An argument based on Chebotarev's density theorem shows that if (2) in the Corollary holds
%then for any $b \in \Z_{\geq 1}$  coprime to $w_{n+1}(\Q)_l$ there exist infinitely many primes $p$ which
%satisfy the conditions in (3) and, in addition, $p$ splits completely in $\Bbb Q(\mu_{l^k})$, where
%$l^k=| K_{2n-1} (\F_{p})_l|.$
%\end{remark}

\begin{remark}
In light of Theorem \ref{generalize Iwasawa} and Corollary
\ref{a different view at the Iwasawa conj.}, the question regarding the cyclicity of
$D(n)_l \in K_{2n} (K)_l$ may be viewed as an extension
of Iwasawa's conjecture to arbitrary totally real number fields $K$, under the obvious hypotheses on $l$ and $n$.
At this point we do not have sufficient data to conjecture that $D(n)_l$ is cyclic at this level of generality. In the next
section we will do a close analysis of the injectivity of $\Lambda_{v_0}$ in the case $F=K=\Bbb Q$. In the process, we will give a new proof of
Proposition \ref{Iwasawa in terms of cyclicity} (see Theorem \ref{another look at the Iwasawa conj.}.) In future work, we hope to extend the techniques developed in the next section to more general totally real number fields $K$
and study a generalization of Iwasawa's cyclicity conjecture in that setting.
\end{remark}
%\bigskip

\section{Conditions for the injectivity of $\Lambda_{v_0}$ in the case $K=F=\Bbb Q$}

In this section we assume that $K=F=\Bbb Q$ and fix an odd prime $l$. Next we pick a natural number ${\bf b}$ which is coprime to $l$ and satisfies
the two conditions in the following elementary Lemma.
\begin{lemma}\label{b-lemma} Let ${\bf b}$ be a natural number coprime to $l$. Then $(1-{\bf b}\cdot\sigma_{\bf b}^{-1})$ is a generator
of the ideal ${\rm Ann}_{\Bbb Z_l[G(\Bbb Q(\mu_l)/\Bbb Q)]}(\mu_l)$ if and only if the following conditions are simultaneously satisfied.
\begin{enumerate}
\item $\omega(\overline{\bf b})$ is a generator of $\mu_{l-1}$, where $\omega: (\Z/l\Z)^\times\to \mu_{l-1}$ is the  Teichm\"{u}ller character and $\overline{\bf b}:={\bf b}\mod l$.
\item ${\bf b}\not\equiv \omega(\overline{\bf b}) \mod l^2\Z_l$.
\end{enumerate}
\end{lemma}
\begin{proof}[Proof (sketch.)] Note that there is a $\Z_l$--algebra isomorphism
$$\Z_l[G(\Bbb Q(\mu_l)/\Bbb Q)]\simeq \bigoplus_{i=0}^{l-2}\Z_l,$$
sending $\sigma\to (\omega^i(\sigma))_i$, for all $\sigma\in\Z_l[G(\Bbb Q(\mu_l)/\Bbb Q)]$. Under this isomorphism
the ideal ${\rm Ann}_{\Bbb Z_l[G(\Bbb Q(\mu_l)/\Bbb Q)]}(\mu_l)$ is sent into $\Z_l\oplus l\Z_l\oplus \Z_l\oplus\dots\oplus \Z_l$. Consequently,
the group ring element $(1-{\bf b}\cdot\sigma_{\bf b}^{-1})$ generates  ${\rm Ann}_{\Bbb Z_l[G(\Bbb Q(\mu_l)/\Bbb Q)]}(\mu_l)$ if and only if
\begin{equation}\label{b-divisibility}l\,||\,(1-{\bf b}\cdot\omega(\overline{\bf b})^{-1}), \qquad l\nmid (1-{\bf b}\cdot\omega(\overline{\bf b})^{-i}), \text{ for all } i\not\equiv 1\mod (l-1)\end{equation}
These are exactly conditions (2) and (1) in the Lemma, respectively.
\end{proof}

\begin{remark}\label{more-b-divisibility}
If ${\bf b}$ is chosen as above, then we also have
$$w_{n+1}(\Q)\Z_l=(1-{\bf b}^{n+1})\Z_l,$$
for all $n\in\Z_{\geq 0}$, as one can easily prove based on relations \eqref{b-divisibility}.
\end{remark}

 Let $n$ be an odd natural number, with $l\nmid n$. Let $v_0$ be a rational prime such that
$$v_0\equiv 1\mod l,\quad v_0\nmid {\bf b}, \quad v_l((1 - {\bf b}^{n+1}) \zeta_{K} (-n))\leq v_l(|K_{2n-1} (k_{v_0})_l|).$$
Note that since $|K_{2n-1} (k_{v_0})_l|=v_0^n-1$, once $n$ and ${\bf b}$ are fixed the set of such primes $v_0$ has positive density, as a consequence of Chebotarev's density theorem.
\medskip

In this context, our goal is to analyze the injectivity of the map
$$
\Lambda_{v_0} \, :\, K_{2n-1}(k_{v_0})_l \rightarrow K_{2n} (\Q)_l
.$$
We resume the notations of \S4. In particular, if $v_l({v_0}^n-1)=k$, then $E:=\Bbb Q(\mu_{l^k})$ and
$w$ is a prime sitting above $v_0$ in $E$. Since in this case the exponent $\gamma_l=(1-l^n)^{-1}$ lies in $\Z_l^\times$ and it does not affect injectivity, in order to simplify notations we work with the following slightly
modified definition of $\Lambda_{v_0}.$
\begin{equation}\label{lambda}\Lambda_{v_0}(\xi_{v_0})=b (Tr_{E/\Q} (\lambda_{{\bf b}, {\bf f}_{E}} (w) \ast
\beta(\xi_{l^k})^{\ast\, n}))=Tr_{E/\Q} (b (\lambda_{{\bf b}, {\bf f}_{E}} (w) \ast
\beta(\xi_{l^k})^{\ast\, n})),\end{equation}
where $\xi_{v_0}$ is the distinguished generator of $K_{2n-1}(k_{v_0})_l$ picked in \S4. Note that since $w\nmid{\bf b}$,
$\lambda^\ast_{{\bf b}, {\bf f}_{E}} (w)=\lambda_{{\bf b}, {\bf f}_{E}} (w)$ in this case.

\medskip

Now, let us note that since $v_l({v_0}-1)>0$,  $v_l({v_0}^n-1)=k$ and $v_l(n)=0$, we have $v_l({v_0}-1)=k$. Consequently, $v_0$ splits completely in $E$ and $q_w={v_0}$. Therefore,
the arguments at the beginning of \S4.2 (see equality \eqref{image-transfer})
imply that all arrows in the following commutative diagram are isomorphisms.
$$\xymatrix{
K_{2n} (k_w, \, \Z/l^k) \ar@{>}[r]^{b_w}_{\cong}\ar@{>}[d]_{Tr_{w/v_0}}^{\cong} &
K_{2n-1} (k_{w})_l \ar@{>}[d]_{Tr_{w/v_0}}^{\cong} \\
K_{2n} (k_{v_0}, \, \Z/l^k) \ar@{>}[r]^{b_{v_0}}_{\cong} & K_{2n-1} (k_{v_{0}})_l}$$

\noindent
Consider the following commutative diagram:

$$\xymatrix{
K_{2n+1} (E, \, \Z/l^k) \ar@{>}[d]_{Tr_{E/K}} &
K_{2n} (k_{w}, \, \Z/l^k) \ar@{>}[l]_{\Lambda_{{w, \, l^k}}} \ar@{>}[d]_{Tr_{w/v_0}}^{\cong} \\
K_{2n+1} (K, \, \Z/l^k) & K_{2n} (k_{v_0}, \, \Z/l^k) \ar@{>}[l]_{\Lambda_{{v_{0}, \,  l^k}}}}$$

\noindent
where the horizontal arrows are defined as follows:

\begin{equation}
\Lambda_{w, \, l^{k}} (\beta(\xi_{l^k})^{\ast\, n}) \,:= \, \lambda_{{\bf b}, {\bf f}_{E}} (w) \ast
\beta(\xi_{l^k})^{\ast\, n},
\label{definition of Lambda w mod lk}
\end{equation}
\begin{equation}
\Lambda_{v_{0}, \, l^{k}} (Tr(\beta(\xi_{l^k})^{\ast\, n})) \, := \,
Tr_{E/\Q} (\lambda_{{\bf b}, {\bf f}_{E}} (w) \ast \beta(\xi_{l^k})^{\ast\, n})
\label{definition of Lambda v0 mod lk}
\end{equation}
\noindent So, we can rewrite
\begin{equation}\label{new expression Lambda}\Lambda_{v_0}(\xi_{v_0})=b\circ Tr_{E/K}\circ\Lambda_{w, l^k}(\beta(\xi_{l^k})^{\ast n})=Tr_{E/\Q}\circ b\circ\Lambda_{w, l^k}(\beta(\xi_{l^k})^{\ast n}).\end{equation}

\medskip

\begin{remark} Note that in the case under consideration, we have
\begin{equation}
\Theta_{0} ({\bf b}, {\bf f}_E) = (1 - {\bf b}\cdot\sigma_{\bf b}^{-1})\cdot \sum_{a\in\left(\Bbb Z/l^k\right)^\times}\zeta_{l^k}(\sigma_a, 0)\cdot\sigma_{a}^{-1} =
\nonumber\end{equation}
\begin{equation}
= (1 - {\bf b}\cdot \sigma_{\bf b}^{-1}) \cdot\sum_{a\in\left(\Bbb Z/l^k\right)^\times} (\frac{1}{2} - \frac{a}{l^k}) \cdot\sigma_{a}^{-1},
\label{Theta 0 for lk}\end{equation}
where the second sum above is taken with respect to all $1\leq a<l^k$, with $l\nmid a$.
The projection of $\Theta_{0} ({\bf b}, {\bf f}_E)$ to $\Bbb Z[G(E_1/\Bbb Q)]$, where $E_1:=\Bbb Q(\mu_l)$,  is given by
$$\Theta_{0} ({\bf b}, {\bf f}_{E_1})=(1 - {\bf b}\cdot \sigma_{\bf b}^{-1}) \cdot \sum_{a\in\left(\Bbb Z/l\right)^\times} (\frac{1}{2} - \frac{a}{l}) \cdot\sigma_{a}^{-1}.$$
Via the first condition in \eqref{b-divisibility}, it is clear that $\omega(\Theta_{0} ({\bf b}, {\bf f}_{E_1}))\not\equiv 0\mod l$. Thus,
$$\Theta_{0} ({\bf b}, {\bf f}_{E_1})\not\in l\Bbb Z_l[G(E_1/\Bbb Q)].$$
This implies in particular that
\begin{equation}\label{Theta 0 for lk}\Theta_{0} ({\bf b}, {\bf f}_{E})\not\in l\Bbb Z_l[G(E/\Bbb Q)].\end{equation}
\end{remark}
\bigskip

Now, we begin our study of the injectivity of $\Lambda_{v_0}$, for the chosen $l, n, {\bf b}, v_0$. We will use
expression \eqref{new expression Lambda} for $\Lambda_{v_0}$. As Steps I and II below will show, the map
$b\circ\Lambda_{w, l^k}$ turns out to be injective unconditionally. The point where Iwasawa's cyclicity conjecture
(Conjecture \ref{Iwasawa conjecture}) comes into play is when one analyzes the injectivity of $Tr_{E/\Q}$ restricted to ${\rm Im}(b\circ\Lambda_{w, l^k})$,
as it will be made clear in Step III below.

\medskip
\noindent
{\bf Step I.} The unconditional injectivity of $\Lambda_{w, \, l^{k}}.$
\medskip

\noindent
By the definition of $\lambda_{{\bf b}, {\bf f}_E} (w)$, its Arakelov divisor is
\begin{equation}
{\rm div}_E(\lambda_{{\bf b}, {\bf f}_E} (w))=(1 - {\bf b}\cdot\sigma_{\bf b}^{-1})) \sum_{a\in(\Bbb Z/l^k)^\times} (\frac{1}{2} - \frac{a}{l^k}) \cdot\sigma_{a}^{-1} (w).
\end{equation}
\noindent Since $v_0$ splits completely in  $E/\Bbb Q$, (\ref{Theta 0 for lk}) implies that
\begin{equation}\label{Arakelov}{\rm div}_E(\lambda_{{\bf b}, {\bf f}_E} (w))\notin \,\,
l\cdot {\rm Div}_{S_\infty}^0(E).
\end{equation}
Consequently, we have
\begin{equation}
\lambda_{{\bf b}, {\bf f}_E} (w) \notin \mu_E\cdot (E^{\times})^{l} \, .
\label{lambda b f not in E l}
\end{equation}
Consider the following commutative diagram.
\begin{equation}\label{Dwyer-Friedlander diagram}\xymatrix{
K_{2n+1} (E, \, \Z/l^k) \ar@{>}[d]_{} &
K_{2n} (k_{w}, \, \Z/l^k) \ar@{>}[l]_{\Lambda_{{w, \, l^k}}} \ar@{>}[d]_{}^{\cong} \\
H^1 (E, \, \Z/l^k (n+1)) & H^0 (k_{w}, \, \Z/l^k (n)) \ar@{>}[l]_{\Lambda_{{w, \,  l^k}}^{et}}}
\end{equation}
The vertical arrows are the Dwyer-Friedlander maps \cite{DFST} and the right vertical arrow
is an isomorphism \cite[Corollary 8.6]{DFST}. The bottom
horizontal arrow is defined by
\begin{equation}
\Lambda_{w, \, l^{k}}^{et} (\xi_{l^k}^{\otimes\, n}) \,:= \, \lambda_{{\bf b}, {\bf f}_E} (w) \cup
\xi_{l^k}^{\otimes\, n}.
\label{definition of Lambda w mod lk}
\end{equation}
Since $\xi_{l^k} \in \mu_E$ it is clear, by (\ref{lambda b f not in E l}), that the map $\Lambda_{w, \, l^{k}}^{et}$ is injective. Hence
the map $\Lambda_{w, \, l^{k}}$ is also injective.

\medskip

\noindent
{\bf Step II.} The unconditional injectivity of the map $(b\circ\Lambda_{w, \, l^{k}}).$
\medskip

\noindent
Consider the following commutative diagram with exact rows and surjective vertical arrows. (Surjectivity
follows from \cite[Theorem 8.5]{DFST}. See also \cite[Thm. 4, p. 278]{So1}.)
\begin{equation}\label{Dwyer-Friedlander diagram 2}\small{\xymatrix{
0 \ar[r]^{} & K_{2n+1} (E) / l^k \ar[r]^{} \ar@{>>}[d]^{}  &
 K_{2n+1} (E, \Z/l^k)  \ar[r]^{b} \ar@{>>}[d]^{} &
K_{2n} (E)[l^k]  \ar@{>>}[d]^{}\ar[r]  & 0\\
0 \ar[r]^{} & H^{1} (E, \, \Z_{l} (n{+}1)/l^k \ar[r]^{} &
 H^{1} (E, \, \Z/l^k (n{+}1))  \ar[r]^{b^{et}} &
 H^{2} (E, \, \Z_{l} (n{+}1)) [l^k]\ar[r]  &0 }}
\end{equation}
\noindent
By Step I the middle vertical arrow in \eqref{Dwyer-Friedlander diagram 2} induces
an isomorphism
\begin{equation}
\text{Im}\, \Lambda_{w, \, l^{k}}\cong
\text{Im}\, \Lambda_{w, \, l^{k}}^{et}.
\label{Im Lambda is isom. to Im Lambda et}
\end{equation}
Hence (\ref{Im Lambda is isom. to Im Lambda et}) and (\ref{Dwyer-Friedlander diagram 2})
give the isomorphism
\begin{equation}\label{Im Lambda is isom. to Im Lambda et itersected with ker b}
\text{Im}\, \Lambda_{w, \, l^{k}} \, \cap \, \text{Ker}\, b \quad \cong \quad
\text{Im}\, \Lambda_{w, \, l^{k}}^{et} \, \cap \, \text{Ker}\, b^{et}.
\end{equation}

Let $S$ denote the set of $l$--adic primes in $E$. Since
$H^1 (\mathcal{O}_{E, S}, \, \Z_l (n+1)) = H^1 (E, \, \Z_l (n+1))$
then $\text{Ker}\, b^{et} \simeq H^1 (\mathcal{O}_{E, S}, \, \Z_l (n+1)) / l^k .$
Hence, isomorphism
(\ref{Im Lambda is isom. to Im Lambda et itersected with ker b})
and (\ref{isom. of plus part of divided cohom. group with twisted roots of unity}) in the Appendix imply
that for an element of the form $\lambda_{{\bf b}, {\bf f}_E} (w)^r \ast \beta(\xi_{l^k})^{\ast\, n}$ we have an equivalence
\begin{equation}\label{kernel}\lambda_{{\bf b}, {\bf f}_E} (w)^r \ast \beta(\xi_{l^k})^{\ast\, n} \in \text{Ker}(b)\Longleftrightarrow \lambda_{{\bf b}, {\bf f}_E} (w)^r \cup \xi_{l^k}^{\otimes\, n} = ({\xi_{l^k}}^{\otimes (n+1)})^s,\end{equation}
for some $s \in \N.$ (When applying \eqref{isom. of plus part of divided cohom. group with twisted roots of unity} it is important to note that since $n$ is odd and $\lambda_{\bf b, f_E}(w)^{1+j}=1$, we have $\lambda_{{\bf b}, {\bf f}_E} (w)^r \cup \xi_{l^k}^{\otimes\, n}\in H^1(E, \Z_l/l^k(n+1))^+$.)

Now, the right hand side of \eqref{kernel} is equivalent to $\lambda_{{\bf b}, {\bf f}_E} (w)^r \in \langle\xi_{l^k}\rangle\cdot E^{\times \, l^k}.$
Combined with \eqref{Arakelov}, this implies that $l^k$ divides $r$. Consequently, $\lambda_{{\bf b}, {\bf f}_E} (w)^r \ast \beta(\xi_{l^k})^{\ast\, n} = 0$. Therefore $\text{Ker}(b\circ\Lambda_{w, l^k})=0$.

\medskip
\noindent
{\bf Step III.} The question of injectivity for the map $Tr_{E/\Q}$ restricted to $\text{Im}\,  (b \circ \Lambda_{w, \, l^{k}}).$
\medskip

\noindent At this point we use a trick which allows us to assume that $k=1$.
First, let us note that $\Lambda_{v_0}$ is injective if and only if
$\Lambda_{v_0}$ restricted to the unique subgroup $K_{2n-1}(k_{v_0})[l]$
of order $l$ in $K_{2n-1}(k_{v_0})_l$ is injective.
Let $$\eta_{v_0} := \xi_{v_0}^{l^{k-1}}.$$
Observe that $\eta_{v_0}$ is a generator of $K_{2n-1}(k_{v_0})[l]$.
Observe that in our case ${\bf f}_{\Q(\mu_l)}=l\Z$ and all the $l$--adic primes in $\Q(\mu_l)$ are
totally ramified in $\Q(\mu_{l^{\infty}}) / \Q(\mu_l).$ Also, recall that $v_0\nmid{\bf b}$. Hence if we apply
Lemma \ref{norm-lemma}, Remark \ref{support of f prime} and Lemma \ref{HilbertChebotarev}(3)
(with $K=F=\Bbb Q$ and $E_1=\Q(\mu_l)$), we obtain:
$$N_{E/E_1}\lambda_{{\bf b}, {\bf f}_{E}} (w) =
\lambda_{{\bf b}, {\bf f}_{E_1}} (w_1)$$
where $w_1$ is a prime of $\mathcal{O}_{E_1}$ below $w.$
Hence
$$\Lambda_{v_0} (\eta_{v_0}) =
\Lambda_{v_0} (\xi_{v_0})^{l^{k-1}} = \, \, b (Tr_{E/\Q} (\lambda_{{\bf b}, {\bf f}_{E}} (w) \ast
\beta(\xi_{l^k})^{\ast\, n}))^{l^{k-1}} = $$
$$ = \,\,
b (Tr_{E_1 / F} (N_{E/ E_1}\lambda_{{\bf b}, {\bf f}_{E}} (w) \ast
\beta(\xi_{l})^{\ast\, n})) \,\, =$$
$$ = \,\, b (Tr_{E_1 / \Q} (\lambda_{{\bf b}, {\bf f}_{E_1}} (w_1) \ast
\beta(\xi_{l})^{\ast\, n})).$$
So the injectivity of $\Lambda_{v_0}$ is equivalent to the injectivity of
the following map:

\begin{equation}
\Lambda_{v_0} \, :\, K_{2n-1}(k_{v_0}) [l] \rightarrow K_{2n} (\Q) [l]
\label{Lambda v0 on [l]}
\end{equation}
\begin{equation}
\Lambda_{v_0} (\eta_{v_0}) :=  b (Tr_{E_1 / \Q} (\lambda_{{\bf b}, {\bf f}_{E_1}} (w_1) \ast
\beta(\xi_{l})^{\ast\, n})).
\label{Lambda v0 (eta v0)}
\end{equation}

If one compares \eqref{Lambda v0 (eta v0)} and \eqref{lambda}, it becomes clear that it is sufficient to set $k =1$ and $E = \Q (\mu_l)$ in all of the above considerations.
We will use this notation for the rest of the section.  Since in this case the natural map
$K_{2n} (\Q)_l \rightarrow K_{2n} (E)_l$ is injective (as $l\nmid |G(E/Q)|$), the injectivity of $\Lambda_{v_0}$ is equivalent to
$${\rm ord}\,(N_l \circ b \circ \Lambda_{w, \, l}) (\beta(\xi_{l})^{\ast\, n})=l,$$
where $N_l := \sum_{c=1}^{l-1} \, \sigma_c.$ The reader should note that we have already showed (in Steps I and II above for $k=1$) that
$${\rm ord}\,(b \circ \Lambda_{w, \, l}) (\beta(\xi_{l})^{\ast\, n})=l.$$

Let $\omega\, :\, G(E/\Q) \rightarrow \Z_{l}^{\times}$ denotes the Teichm{\" u}ller character and let $e_{\omega^i}\in\Bbb Z_l[G]$ denote the
group ring idempotent corresponding to $\omega^i$, for all $i\in\Bbb Z$. Note that
\begin{equation}
e_{\omega^i} \equiv \sum_{c=1}^{l-1} \, - \, c^i \, \sigma_{c}^{-1} \mod l\Bbb Z_l[G].
\end{equation}
Also, note that $E^\times/E^{\times l}$ has a natural $\Bbb Z_l[G]$--module structure. For simplicity, if $x\in E^\times$, we will
use the notation
$$e_{\omega^i}\cdot x:= e_{\omega^i}\cdot (x\mod E^{\times l})$$
and view this as an element in $E^\times/E^{\times l}$, for all $i\in \Bbb Z$.
In light of the above notation, we have the following equalities:
\begin{equation}
N_l ( \Lambda_{w, \, l} (\beta(\xi_{l})^{\ast\, n}) ) =
\sum_{c=1}^{l-1} \, \sigma_c (\lambda_{{\bf b}, {\bf f}_{E}} (w) \ast
\beta(\xi_{l})^{\ast\, n}) =
\label{product of Gauss sums times power of Bott}\end{equation}
\begin{equation}
 = \bigl(\prod_{c=1}^{l-1}
\lambda_{{\bf b}, {\bf f}_{E}} (\sigma_{c} (w))^{c^n}\bigr) \ast
\beta(\xi_{l})^{\ast\, n} = - e_{\omega^{- n}} (\lambda_{{\bf b}, {\bf f}_{E}} (w)) \ast
\beta(\xi_{l})^{\ast\, n}.
\nonumber\end{equation}

\begin{equation}
N_l ( \Lambda_{w, \, l}^{et} (\xi_{l}^{\otimes\, n}) ) =
\sum_{c=1}^{l-1} \, \sigma_c (\lambda_{{\bf b}, {\bf f}_{E}} (w) \cup
\xi_{l}^{\otimes\, n}) =
\label{product of Gauss sums times power of root of unity}\end{equation}
\begin{equation}
 = \bigl(\prod_{c=1}^{l-1}
\lambda_{{\bf b}, {\bf f}_{E}} (\sigma_{c} (w))^{c^n}\bigr) \cup
\xi_{l}^{\otimes \, n} = - e_{\omega^{- n}} (\lambda_{{\bf b}, {\bf f}_{E}} (w)) \cup
\xi_{l}^{\otimes \, n}.
\nonumber\end{equation}
Now, \eqref{product of Gauss sums times power of Bott} and \eqref{product of Gauss sums times power of root of unity} combined with the arguments at the
end of Step II (applied to $k:=1$) lead to the following equivalence.
\begin{equation}\label{kernel-characters} N_l ( \Lambda_{w, \, l} (\beta(\xi_{l})^{\ast\, n}) )\in\text{Ker}(b) \Longleftrightarrow e_{\omega^{- n}} (\lambda_{{\bf b}, {\bf f}_{E}} (w))\in (\mu_l\cdot E^{\times l})/E^{\times l}.\end{equation}
Now, we are ready to prove the main result of this section.
\begin{theorem} \label{another look at the Iwasawa conj.} Let $l > 2$ be a prime number. Let $n\geq 1$ be an odd integer, such that $l\nmid n$.
Let ${\bf b}\geq 1$ be an integer satisfying the
two conditions in Lemma \ref{b-lemma}. As above, let $E:=\Q(\mu_l)$. Then, the following are equivalent.
\begin{enumerate}
\item
$D (n)_l:=div\left(K_{2n}(\Q)_l\right)$ is cyclic.
\item
 There exists an $O_E$--prime $w$ for which
 $$e_{\omega^{- n}} (\lambda_{{\bf b}, {\bf f}_{E}} (w))\not\in(\mu_l\cdot E^{\times l})/E^{\times l}.$$
 and which satisfies the following additional hypotheses: it is coprime to ${\bf b}$, it sits above a rational prime $p$ which splits completely in $E / \Q$, and
$$|(1-{\bf b}^{n+1})\cdot \zeta_{\Bbb Q} (-n)|^{-1}_{l} \,\, | \,\, p^n - 1.$$
\item $A^{[l-1-n]}:=e_{\omega^{-n}}\cdot(Cl(\CO_E)\otimes\Z_l)$ is cyclic.
\end{enumerate}
 \end{theorem}
\begin{proof} {\bf (1) $\Longleftrightarrow$ (2).} Remark \ref{more-b-divisibility}
and Lemma \ref{infinitely many v} applied to $(\Q/\Q, n, l, {\bf b})$ show that $D(n)_l$ is cyclic if and only if there exists a
rational prime satisfying the additional hypotheses in (2) for which the map $\Lambda _p$ associated
to $(\Q/\Q, n, l, {\bf b}, p)$ is injective.

Now, since for an $\CO_E$--prime $w|p$, we have an equality
$$(N_l \circ b \circ \Lambda_{w, \, l}) (\beta(\xi_{l})^{\ast\, n}) =
b (N_l ( \Lambda_{w, \, l} (\beta(\xi_{l})^{\ast\, n}))),$$
\eqref{product of Gauss sums times power of Bott} and \eqref{kernel-characters} imply that the injectivity of $\Lambda_p$ holds if and only if
$$e_{\omega^{- n}} (\lambda_{{\bf b}, {\bf f}_{E}} (w))\not\in(\mu_l\cdot E^{\times l})/E^{\times l},$$
for some $\CO_E$--prime $w$ sitting above $p$.
\medskip

{\bf (2) $\Longleftrightarrow$ (3).} Our choice of ${\bf b}$ (see Lemma 6.1) combined with a theorem of Mazur--Wiles (see Theorem 8.8 and the Remark which follows in the Appendix of \cite{Lang}) gives us the following equalities.
\begin{equation}\label{strong Kummer-Ribet}
\mid A^{[l-1-i]}\mid =[\Z_l: \omega^{-i}(\Theta_0({{\bf b}, {\bf f}_E}))\Z_l], \qquad\text{ for all odd $i$. }
\end{equation}
The reader should note that our choice of ${\bf b}$ implies that $\Theta_0({{\bf b}, {\bf f}_E})$ is a generator of the Stickelberger ideal in $\Z_l[G(E/\Q)]$, denoted by $\mathcal R_0$ in
the Remark following Theorem 8.8 in loc.cit. Also, the equality in the Remark in question can be easily extended to the case $\chi:=\omega$ as both sides are equal to $1$ in that case.

We consider the standard exact sequence of $\Z_l[G(E/\Q)]$--modules
$$\xymatrix{ 0\ar[r] &(E^\times/\CO_E^\times)\otimes\Z_l \ar[r]^{{\rm div}{\,\,}} &\ar[r] {\rm Div}(\CO_E)\otimes\Z_l\ar[r] & Cl(\CO_E)\otimes\Z_l\ar[r] &0,
}$$
where ${\rm Div}(\CO_E)$ is the non-archimedean part of the group ${\rm Div}_{S_\infty}(E)$ of Arakelov divisors of $E$, ${\rm div}$ is the  divisor map extended by $\Z_l$--linearity
to $(E^\times\otimes\Z_l)$  and the projection onto $Cl(O_E)\otimes\Z_l$ is the usual divisor--class map (taking $d\in {\rm Div}(\CO_E)$ to its class $\widehat d\in Cl(\CO_E)$) extended by $\Z_l$--linearity. Since $n$ is odd, we have an equality
$e_{\omega^{-n}}\cdot (\CO_E^\times\otimes\Z_l)=e_{\omega^{-n}}\cdot(\mu_l\otimes\Z_l)$. Consequently, by taking $\omega^{-n}$--components in the exact sequence above we obtain the following exact sequence of $\Z_l$--modules
\begin{equation}\label{omega-sequence}\xymatrix{ 0\ar[r] &e_{\omega^{-n}}(E^\times/\mu_l\otimes\Z_l) \ar[r]^{{\rm div}{\,\,}} &\ar[r] e_{\omega^{-n}}({\rm Div}(\CO_E)\otimes\Z_l)\ar[r] & A^{[l-1-n]}\ar[r] &0.
}\end{equation}

Now, according to \eqref{strong Kummer-Ribet}, $A^{[l-1-n]}$ is cyclic if and only if it contains an ideal class of order $|\omega^{-n}(\Theta_0({{\bf b}, {\bf f}_E}))|_l^{-1}$. Consequently, exact sequence \eqref{omega-sequence}
and Lemma \ref{HilbertChebotarev} applied to $E/\Q$ and a sufficiently large $m$, imply that $A^{[l-1-n]}$ is cyclic if and only if there exists an $O_E$--prime $w$ satisfying the additional hypotheses in (2), such that
\begin{equation}\label{ciclicity-divisors}
\omega^{-n}(\Theta_0({{\bf b}, {\bf f}_E}))\cdot e_{\omega^{-n}}(w\otimes 1)\in {\rm div}(e_{\omega^{-n}}(E^\times\otimes\Z_l))\setminus l\cdot {\rm div}(e_{\omega^{-n}}(E^{\times}\otimes\Z_l)).
\end{equation}
In the argument above, it is important to note that since $w$ sits over a rational prime $p$ which splits completely in $E/\Q$, we have $e_{\omega^{-n}}(w\otimes 1)\not\in l\cdot ({\rm Div}(\CO_E)\otimes\Z_l).$
However, by the definition of
$\lambda_{{\bf b}, {\bf f}_{E}} (w)$,
we have
\begin{equation}\label{divisor-lambda}{\rm div}(e_{\omega^{-n}}(\lambda_{{\bf b}, {\bf f}_{E}} (w)\otimes 1))=\omega^{-n}(\Theta_0({{\bf b}, {\bf f}_E}))\cdot e_{\omega^{-n}}(w\otimes 1).\end{equation}
Equality \eqref{divisor-lambda} combined with exact sequence \eqref{omega-sequence} shows that \eqref{ciclicity-divisors} is equivalent to
$$e_{\omega^{-n}}(\lambda_{{\bf b}, {\bf f}_{E}} (w)\otimes 1)\not\in e_{\omega^{-n}}(\mu_l\cdot E^{\times l}\otimes\Z_l).$$
This is clearly equivalent to $e_{\omega^{-n}}(\lambda_{{\bf b}, {\bf f}_{E}} (w))\not\in (\mu_l\cdot E^{\times l})/E^{\times l}$.
\end{proof}
\begin{remark} Note that for $n\not\equiv -1\mod (l-1)$, we have an equivalence
$$ e_{\omega^{- n}} (\lambda_{{\bf b}, {\bf f}_{E}} (w))\not\in(\mu_l\cdot E^{\times l})/E^{\times l}\Longleftrightarrow e_{\omega^{- n}} (\lambda_{{\bf b}, {\bf f}_{E}} (w))\ne 0,$$
since we have an obvious inclusion $(\mu_l\cdot E^{\times l})/E^{\times l}\subseteq e_{\omega^1}\cdot (E^\times/E^{\times l}).$
\end{remark}

%-----------\section{The Euler system}------------------

\section{An Euler system in the higher odd $K$--theory with coefficients}

In this section we construct an Euler System for the odd $K$-theory (with coefficients) of a CM abelian
extension of a totally real number field. Our construction generalizes those of \cite{Rubin} and \cite{BG1} to arbitrary totally real number fields
and it is quite different from that in \cite{BP}.
\medskip

As above, we fix a finite abelian CM extension $F/K$ of a totally real number field of conductor ${\bf f}$, a rational prime $l>2$ and a natural number $n\geq 1$.
Next, we fix an $\CO_K$--ideal ${\bf b}$ which is coprime to ${\bf f}l$.
We let ${\bf L} = {\bf l}_1 \cdot\dots\cdot {\bf l}_t$ run through all the products of mutually distinct
prime ideals ${\bf l}_1, \dots,{\bf l}_t$ of ${\mathcal O}_K$, coprime to $l\cdot{\bf bf}.$
Let
$F_{\bf L} := F K_{\bf L}$, where
$K_{\bf L}$ is the ray class field of $K$ of conductor ${\bf L}.$
Obviously, the conductor of the CM-extension $F_{\bf L} / K$
divides ${\bf L} {\bf f}.$ We let $F_{{\bf L}{l^k}} := F_{{\bf L}}(\mu_{l^k})$, for every $k \geq 0.$
\medskip

For each CM extension
$F_{{\bf L}{l^{k}}}/K$ we fix an $\CO_K$--ideal ideal ${\bf f}_{{\bf L}{l^{k}}}'$
such that
$${\rm Supp}({\bf f}_{{\bf L}{l^{k}}}') \,\, = \,\, {\rm Supp}({\bf f}) \cup {\rm Supp}({\bf L}) \cup {\rm Supp}(l\CO_K).$$
Also, we fix roots of unity $\xi_{l^k}\in F(\mu_{l^k})$ of order $l^k$, such that
$$\xi_{l^{k+1}}^l=\xi_{l^k},$$
for all $k\geq 0$. We let
$$\beta_{{\bf L}, k}:=\beta_{F_{{\bf L}l^k}}(\xi_{l^k})$$
be the corresponding Bott elements in $K_2(F_{{\bf L}l^k}, \Z/l^k)$, for all ${\bf L}$ and $k$ as above.
Next, we fix a prime $v$ in ${\mathcal O}_{F}$, such that $v \nmid l{\bf b}$.
For each $\bf L$ as above and each $k\geq 0$, we fix a prime $w_{k}(\bf L)$ of $O_{F_{{\bf L}l^k}}$ sitting above
$v$, such that $w_{k'}(\bf L')$ sits above $w_{k}(\bf L)$ whenever $l^k{\bf L}\mid l^{k'}{\bf L'}$. We let $v({\bf L}):=w_0({\bf L})$,
for all ${\bf L}$ as above.
\medskip

Let ${\bf f}^\ast_{{\bf L}{l^{k}}}$ denote the conductor of $F_{{\bf L}l^k}/K$ multiplied by all the $l$--adic primes in $K$. We will use the $l$--adic imprimitive Brumer--Stark elements $\{\lambda_{{\bf b}, {{\bf f}^\ast_{{\bf L}{l^{k}}}}}(w_{k}({\bf L}))\}_{\bf L}$ viewed as special elements in $\{K_1(F_{{\bf L}{l^{k}}})_l\}_{\bf L}$ to construct special elements
$\{\blambda_{v({\bf L}), k}\}_{\bf L}$ in the $K$--theory with coefficients $\{K_{2n+1}(F_{\bf L},
\Bbb Z/l^k)\}_{\bf L}$, for all $k>0$ as follows.

\begin{definition}\label{define Euler system} For all ${\bf L}$ and $k$ as above and $n\geq 0$ define
$\blambda_{v({\bf L}), k} \in
K_{2n +1} (F_{\bf L}, \, \Z/l^k)$ as follows.
$$\blambda_{v({\bf L}), k} \,\, := \,\,
Tr_{F_{{\bf L}{l^{k}}} / F_{\bf L}} ( \lambda_{{\bf b}, {{\bf f}^\ast_{{\bf L}{l^{k}}}}}(w_{k}({\bf L})) \ast
\beta_{{\bf L}, k}^{\ast \, n})^{ \gamma_{l}},
$$
where the operator $\gamma_l\in\Bbb Z_l[G(F_{\bf L}/K)]$ is given by \eqref{gamma_l}.
\end{definition}
\noindent The reader should note that in the definition above, for simplicity, our notation does not capture the dependence of $\gamma_{l}$ on ${{\bf L}}$.
\medskip

In what follows, we will apply the following three natural maps in K-theory: the transfer map
(\ref{def of norm on K groups of finite fields}), the reduction  modulo $w_k({\bf L})$ map (\ref{reduction map}), the
boundary map in the Quillen localization sequence for $K$--theory with coefficients (\ref{boundary map}) and the reduction modulo $l^{k}$ map
\eqref{reduction mod l map}:
\begin{eqnarray}
\label{def of norm on K groups of finite fields}
Tr_{w_{k}({\bf L})/w_{0}({\bf L})} \, :\, K_{2 n} (k_{w_{k}(\bf L)}, \, \Z/l^k) \rightarrow
K_{2 n} (k_{w_{0}(\bf L)}, \, \Z/l^k)\\
\label{reduction map}
\pi_{w_{k}({\bf L})} \, :\, K_{2 n} ({\mathcal O}_{F_{{\bf L}{l^k}}, S}, \, \Z/l^k) \,
\rightarrow  K_{2 n} (k_{w_{k}(\bf L)}, \, \Z/l^k) \\
\label{boundary map}
\partial_{F_{{\bf L}}} \, : \, K_{2 n+1} (F_{{\bf L}}, \, \Z/l^k) \,\,
\stackrel{}{\longrightarrow} \,\, \bigoplus_{v} K_{2 n} (k_v, \, \Z/l^k)\\
\label{reduction mod l map}
r_{k'/k}: K_{2n +1} (F_{\bf L}, \, \Z/l^{k'})\to K_{2n +1} (F_{\bf L}, \, \Z/l^k), \qquad\text{ for all }k'\geq k.
\end{eqnarray}
The following theorem captures the main properties of the special elements defined above. In particular, part (3) of the theorem below
simply states the fact that these special elements form an Euler system in $\{K_{2n+1}(F_{\bf L},
\Bbb Z/l^k)\}_{\bf L}$, for all $k>0$.
\begin{theorem}\label{Theorem ES2}
Assume that ${\bf L}^{\prime} =
{\bf L} {\bf l}^{\prime}$, for an $\CO_K$--prime ${\bf l'}\nmid {\bf L}$. Let us denote
$N_{{\bf L}} := Tr_{k_{w_k({\bf L})}/k_{w_0({\bf L})}}.$
\begin{itemize}

\item[{(1)}] If $k^{\prime} \geq k$ and  $w_{k}({\bf L})$ splits completely in $F_{{\bf L}l^{k'}},$ then:
$$r_{k^{\prime}/k} (\blambda_{v({\bf L}), k'}) =
\blambda_{v({\bf L}), k}.$$
\item[{(2)}]
$\partial_{F_{{\bf L}}} (\blambda_{v({\bf L}), k}) =
N_{{\bf L}}(
\pi_{w_{k}({\bf L})} ( \beta_{{\bf L}, k}^{\ast \, n}))^{l^{v_{l} (n)} \Theta_{n} ({\bf b},\,  {\bf L  f})}.$ \quad
\item[{(3)}]
If $\sigma_{{\bf l'}}$ denotes the Frobenius morphism associated to ${\bf l'}$ in $G(F_{\bf L}/K)$, then
$$ Tr_{F_{{\bf L}^{\prime}}/F_{\bf L}}
(\blambda_{v ({\bf L}^{\prime}), k}) =
(\blambda_{v({\bf L}), k})^{1 - N({\bf l}^{\prime})^{n}
\cdot \sigma_{{\bf l'}}^{-1}}.$$
\end{itemize}

\end{theorem}

\begin{proof} Part (1)
follows by Lemma \ref{norm-lemma} and the projection formula.
The proof part (2) is similar to the proof of Theorem
\ref{local-lambda-theorem}(2).
Let us prove the third formula above (the Euler system property).
We apply Definition \ref{define Euler system} and Lemma \ref{norm-lemma} combined with Remark \ref{norm-extend-remark} to conclude that we have the following equalities:

\begin{eqnarray}
\nonumber & Tr_{F_{{\bf L}^{\prime}}/F_{{\bf L}}}
(\blambda_{v({\bf L'}), k}) = Tr_{F_{{\bf L}^{\prime}}/F_{{\bf L}}}
Tr_{F_{{{\bf L}^{\prime}}{l^{k}}} / F_{{\bf L}^{\prime}}} ( \lambda_{{\bf b, {{\bf f}^\ast_{{\bf L^{\prime}}{l^{k}}}}}}(w_{k}({\bf L}^{\prime})) \ast
\beta_{{\bf L'}, k}^{\ast \, n})^{ \gamma_{l}} = \\
\nonumber  & Tr_{F_{{{\bf L} l^k}} /F_{{\bf L}}}
Tr_{F_{{\bf L}^{\prime} l^k}/F_{{\bf L} l^k}}
(\lambda_{{\bf b}, {{\bf f}^\ast_{{\bf L^{\prime}}{l^{k}}}}}(w_{k}({\bf L}^{\prime})) \ast
\beta_{{\bf L'}, k}^{\ast \, n})^{ \gamma_{l}}
=\\
\nonumber & Tr_{F_{{{\bf L} l^k}} /F_{{\bf L}}} (N_{F_{{\bf L}^{\prime} l^k}/F_{{\bf L} l^k}}
\lambda_{{\bf b}, {{\bf f}^\ast_{{\bf L^{\prime}}{l^{k}}}}}(w_{k}({\bf L}^{\prime})) \ast
\beta_{{\bf L}, k}^{\ast \, n})^{ \gamma_{l}}
=\\
\nonumber & Tr_{F_{{{\bf L} l^k}} /F_{{\bf L}}}(
\lambda_{{\bf b}, {{\bf f}^\ast_{{\bf L}{l^{k}}}}}(w_{k}({\bf L}))^{1 - {\sigma_{{\bf l'}}}^{-1}} \ast
\beta_{{\bf L}, k}^{\ast \, n})^{ \gamma_{l}}
=\\
\nonumber &  (Tr_{F_{{{\bf L} l^k}} /F_{{\bf L}}}(
\lambda_{{\bf b}, {{\bf f}^\ast_{{\bf L}{l^{k}}}}}(w_{k}({\bf L})) \ast
\beta_{{\bf L}, k}^{\ast \, n})^{ \gamma_{l}})^{1 - N({\bf l'})^{n}\cdot{\sigma_{{\bf l'}}}^{-1}}
= \\
&(\blambda_{v({\bf L}), k})^{1 - N({\bf l}^{\prime})^{n}
\cdot\sigma_{\bf l'}^{-1}}.
\nonumber
\end{eqnarray}
\end{proof}
\begin{remark}\label{Euler-system-remark} Note that if $K=\Q$, then our Euler system construction recaptures that of \cite{BG1}.
Also, if $n=0$, our Euler system lives in $\{F_{\bf L}^\times/F_{\bf L}^{\times l^k}\}_{\bf L}$ and it is given by
$$\blambda_{v({\bf L}), k} \,\, := \,\,
Tr_{F_{{\bf L}{l^{k}}} / F_{\bf L}} ( \lambda_{{\bf b}, {{\bf f}^\ast_{{\bf L}{l^{k}}}}}(w_{k}({\bf L})) )^{ \gamma_{l}} \mod F_{\bf L}^{\times l^k}.
$$
This is a vast generalization of Kolyvagin's Euler system of Gauss sums (mod $l^k$) (see \cite{Rubin}) which can be obtained from our
construction by setting $K=\Q$.
\end{remark}
\bigskip

\section{Appendix: The class groups and \'etale cohomology groups of  rings of $S$-integers}
\bigskip

\noindent
All the cohomology groups in what follows are \'etale cohomology groups. The following two lemmas are very useful in section 6.

\begin{lemma} \label{class group and cohomology}
Let $l$ be an odd prime. Let $L$ be a number field such that there is only one prime $v_l$ over $l$ in $\mathcal{O}_L.$
Let $S :=S_l= \{v_l\}$.
Then for any $m \in \N$ we have the following natural isomorphisms:
\begin{equation}
Cl (\mathcal{O}_{L, S})/l^m \stackrel{\cong}{\longrightarrow} H^2 (\mathcal{O}_{L, S}, \, \Z/l^m (1))
\label{Cl Z mod lm}\end{equation}
\begin{equation}
Cl (\mathcal{O}_{L, S})_l \stackrel{\cong}{\longrightarrow} H^2 (\mathcal{O}_{L, S}, \, \Z_l (1))
\label{Cl Zl}\end{equation}
\begin{equation}
Cl (\mathcal{O}_{L, S}) [l^m] \stackrel{\cong}{\longrightarrow} H^2 (\mathcal{O}_{L, S}, \, \Z_l (1)) [l^m]
\label{Cl lm truncation}\end{equation}
\end{lemma}

\begin{proof} It is well known (see \cite[\S2]{K} and the references therein) that we have the following canonical isomorphisms
\begin{eqnarray}
% \nonumber to remove numbering (before each equation)
\label{units}  H^0(\CO_{L,S}, \G_m) &\simeq & \CO_{L,S}^\times \\
\label{class group}  H^1(\CO_{L,S}, \G_m) &\simeq &  Cl (\mathcal{O}_{L, S}) \\
\label{Brauer group}  H^2(\CO_{L,S}, \G_m)&\simeq &Br(\CO_{L,S})\simeq (\Q/\Z)^{|S|-1}
\end{eqnarray}
where $\G_m$ is the \'etale sheaf associated to the multiplicative group scheme $\G_m$ and $Br$ stands for Brauer group.
Since in our case $|S|=1$, we have $H^2(\CO_{L,S}, \G_m)=0$. Therefore, the long exact sequence in \'etale cohomology associated to the short exact
sequence of \'etale sheaves on $\text{Spec} (\mathcal{O}_{L, S})$
\begin{equation}\label{Kummer sequence}1 \rightarrow \mu_{l^m} \rightarrow \G_m \stackrel{l^m}{\rightarrow} \G_m \rightarrow 1\end{equation}
combined with \eqref{class group} gives the isomorphism
$$Cl (\mathcal{O}_{L, S})/l^m \simeq H^2 (\mathcal{O}_{L, S}, \, \mu_{l^m})=H^2 (\mathcal{O}_{L, S}, \, \Z/l^m(1)),$$
which is precisely \eqref{Cl Z mod lm}. Isomorphism
(\ref{Cl Zl}) follows upon taking an inverse limit with respect to $m$. Isomorphism
(\ref{Cl lm truncation}) follows
upon applying the multiplication by $l^m$ map to both sides of \eqref{Cl Zl}.
\end{proof}

\begin{lemma} Let $E := \Q(\mu_{l^k})$, for some $k\geq 0$.
Let $v_l$ be the unique prime of $\mathcal{O}_{E}$ over $l$ and let $S:=\{v_l\}$. Assume that $n$ is odd.
Then, there are natural isomorphisms of $\Z/l^k [G(E/\Q)]$ modules:
\begin{equation}
(H^1 (\mathcal{O}_{E, S}, \, \Z_l (n+1)) / l^k)^{+} \,\,\cong \,\, \Z/l^k (n+1)
\label{isom. of plus part of divided cohom. group with twisted roots of unity}
\end{equation}
\begin{equation}
Cl (\mathcal{O}_{E, S})\, [l^k]^{-} \otimes  \Z/l^k (n) \,\,\cong
\,\, H^2 (\mathcal{O}_{E, S}, \Z_l (n+1))^{+}\, [l^k]
\label{isom. of plus part of twisted class group with plus part of cohom.}
\end{equation}
Above, the upper scripts $\pm$ stand for the corresponding eigenspaces with respect to the action of the unique
complex conjugation automorphism of $E$.
\end{lemma}

\begin{proof}
Since $n+1\geq 2$ and $(n+1)$ is even, if we combine \cite[Prop. 2.9]{K} with \cite[p. 239]{K} we obtain the following natural isomorphisms of $\zl[G(E/\Q)]$--modules
$$H^1 (\mathcal{O}_{E, S}, \, \Z_l (n+1))^+\simeq H^1 (\mathcal{O}_{E^+, S}, \, \Z_l (n+1))\simeq [\ql/\zl(n+1)]^{G_{E^+}},$$
where $E^+$ is the maximal real subfield of $E$ and $G_{E^+}$ is its absolute Galois group. Since $2=[E^+(\mu_{l^k}):E^+]$ divides $n+1$,
we have $[\ql/\zl(n+1)]^{G_{E^+}}\simeq\Z/l^\alpha(n+1)$, for some $\alpha\geq k$. Consequently, we have
$$(H^1 (\mathcal{O}_{E, S}, \, \Z_l (n+1)) / l^k)^{+}\simeq H^1 (\mathcal{O}_{E, S}, \, \Z_l (n+1))^+/l^k\simeq \Z/l^k(n+1),$$
which proves isomorphism \eqref{isom. of plus part of divided cohom. group with twisted roots of unity}.
\medskip

Now, let $E_\infty:=E(\mu_{l^\infty})$ and $\Gamma:=G(E_\infty/E)$. Obviously, $\Gamma\simeq G(E^+_\infty/E^+)$. Since the $l$--adic primes
in $E, E^+, E_\infty, E_\infty^+$ are principal, we have $Cl(O_{E,S})=Cl(O_{E})$ and also $Cl(O_{E_\infty,S})=Cl(O_{E_{\infty}})$. Similar equalities hold for $E^+$ and $E_{\infty}^+$.
Moreover, since $\l\ne 2$, the natural map at the level of ideal classes induces an injection
$$Cl(O_{E})_l^-\subseteq Cl(O_{E_\infty})_l^-.$$
(See \cite[Prop. 13.26]{Washington}.) It is also well known (as a direct consequence of the cohomological triviality of $Cl(O_{E(\mu_{l^m})})_l^-$ as a $G(E(\mu_{l^m})/E)$--module, for all $m$, see \cite{Greither})
that the inclusion above induces an equality
\begin{equation}\label{class invariants}Cl(O_{E})_l^- = [Cl(O_{E_\infty})_l^-]^\Gamma.\end{equation}
Now, if $\frak X_\infty^+$ denotes the Galois group of the maximal pro-$l$ abelian extension of $E_\infty^+$ which is unramified outside of $l$, \cite[Prop. 2.20]{K} and \cite[p. 238]{K} give natural isomorphisms
of $\zl[G(E/\Q)]$--modules
$$H^2 (\mathcal{O}_{E, S}, \Z_l (n+1))^{+}\simeq H^2 (\mathcal{O}_{E^+, S}, \Z_l (n+1))\simeq (\frak X_\infty^+(-(n+1))_\Gamma)^\vee,$$
where $\ast_{\Gamma}$ stands for $\Gamma$--coinvariants and ${\ast}^\vee$ stands for Pontrjagin dual. However, we have a natural perfect duality
pairing (see \cite[Prop. 13.32]{Washington})
$$\frak X_\infty^+\times Cl(O_{E_\infty})_l^-\to \ql/\zl(1).$$
When combined with the last displayed isomorphisms this pairing leads to
$$H^2 (\mathcal{O}_{E, S}, \Z_l (n+1))^{+}\simeq [Cl(O_{E_\infty})_l^-(n)]^\Gamma.$$
When combined with \eqref{class invariants} this isomorphism leads to
\begin{eqnarray}
% \nonumber to remove numbering (before each equation)
 \nonumber H^2 (\mathcal{O}_{E, S}, \Z_l (n+1))^{+}[l^k] &\simeq & [Cl(O_{E_\infty})_l[l^k]^-(n)]^\Gamma\\
 \nonumber  &\simeq & [Cl(O_{E_\infty})_l[l^k]^-\otimes\Z/l^k(n)]^\Gamma \\
 \nonumber  &=& Cl(O_{E})_l[l^k]^-\otimes\Z/l^k(n).
\end{eqnarray}
The last equality above is a consequence of the fact that $\Gamma$ acts trivially on $Z/l^k(n)$ (as $\mu_{l^k}\subseteq E^\times$.) This concludes the
proof of isomorphism \eqref{isom. of plus part of twisted class group with plus part of cohom.}.
\end{proof}

%-----------------BIBLIOGRAPHY----------

\bibliographystyle{amsplain}

\end{document}